\documentclass[12pt,a4paper,leqno]{amsart}
\usepackage{amsfonts,amssymb,amsmath,amsthm,graphicx,color,amscd,xspace,verbatim}
\usepackage{hyperref}
\usepackage[margin=1in]{geometry}

\newtheorem{theorem}{Theorem}[section]
\newtheorem{proposition}[theorem]{Proposition}

\newtheorem{lemma}[theorem]{Lemma}
\theoremstyle{definition}

\newtheorem{definition}[theorem]{Definition}
\numberwithin{equation}{section}

\newcommand{\be}{\begin{equation}}
\newcommand{\bel}[1]{\begin{equation}\label{#1}}
\newcommand{\ee}{\end{equation}}

\newcommand{\BA}{\begin{array}}
\newcommand{\EA}{\end{array}}

\newcommand{\BAL}{\begin{aligned}}
\newcommand{\EAL}{\end{aligned}}

\newcommand{\Proof}{\note{Proof}}
\newcommand{\note}[1]{\noindent\textit{#1.}\hskip 2mm}

\newcommand{\forevery}{\quad \forall}

\newcommand{\abs}[1]{\left |#1\right |}
\newcommand{\norm}[1]{\left \|#1\right \|}

\newcommand{\dist}{\mathrm{dist}\,}
\newcommand{\sign}{\mathrm{sign}}

\newcommand{\q}{\quad}
\newcommand{\qq}{\qquad}

\newcommand{\prt}{\partial}
\newcommand{\sms}{\setminus}
\newcommand{\ti}{\times}
\newcommand{\tl}{\tilde}
\newcommand{\sbs}{\subset}
\newcommand{\ity}{\infty}

\def\ga{\alpha}     \def\gb{\beta}       \def\gg{\gamma}
       \def\gd{\delta}      \def\ge{\epsilon}
\def\gth{\theta}                         \def\vge{\varepsilon}
\def\gf{\phi}       \def\vgf{\varphi}    
      \def\gk{\kappa}      \def\gl{\lambda}
\def\gm{\mu}        \def\gn{\nu}         \def\gp{\pi}
            
\def\gs{\sigma}       \def\gt{\tau}
      \def\gw{\omega}
                \def\gz{\zeta}
     \def\Gd{\Delta}      

\def\Gl{\Lambda}    \def\Gs{\Sigma}      
\def\Gw{\Omega}              

\def\CS{{\mathcal S}}      
   \def\CO{{\mathcal O}}   
      
      \def\CF{{\mathcal F}}
      
   \def\CK{{\mathcal K}}   \def\CL{{\mathcal L}}

\def\BBA {\mathbb A}       
       
\def\BBG {\mathbb G}   \def\BBH {\mathbb H}    
   \def\BBK {\mathbb K}    
   \def\BBN {\mathbb N}    
   \def\BBR {\mathbb R}

\def\GTM {\mathfrak M}

\def\tr{\mathrm{tr\,}}
\def\div{\mathrm{div}}

\newcommand{\ei}{\phi_{\xm }}
\newcommand{\xa}{\alpha}
\newcommand{\xb}{\beta}
\newcommand{\xg}{\gamma}
\newcommand{\xG}{\Gamma}
\newcommand{\xd}{\delta}
\newcommand{\xD}{\Delta}
\newcommand{\xe}{\varepsilon}

\newcommand{\xk}{\kappa}
\newcommand{\xl}{\lambda}
\newcommand{\xL}{\Lambda}
\newcommand{\xm}{\mu}
\newcommand{\xn}{\nu}

\newcommand{\xs}{\sigma}

\newcommand{\xf}{\phi}

\newcommand{\xo}{\omega}
\newcommand{\xO}{\Omega}

\begin{document}

\title[Semilinear elliptic equations with Hardy potential]{Semilinear elliptic equations with Hardy potential and gradient nonlinearity}
\author{Konstantinos Gkikas}
\address{Konstantinos T. Gkikas, Department of Mathematics, National and Kapodistrian University of Athens, 15784 Athens, Greece}
\email{kugkikas@gmail.com}
\author{Phuoc-Tai Nguyen}
\address{P. T. Nguyen, Department of Mathematics ans Statistics, Masaryk University, Brno, Czech Republic.}
\email{ptnguyen@math.muni.cz}


\maketitle

\tableofcontents
\begin{abstract}
	Let $\Omega \subset {\mathbb R}^N$ ($N \geq 3$) be a $C^2$ bounded domain and $\delta$ be the distance to $\partial \Omega$. We study positive solutions of equation (E) $-L_\mu u+ g(|\nabla u|) = 0$ in $\Omega$ where $L_\mu=\Delta + \frac{\mu}{\delta^2} $, $\mu \in (0,\frac{1}{4}]$ and $g$ is a continuous, nondecreasing function on ${\mathbb R}_+$. We prove that if $g$ satisfies a singular integral condition then there exists a unique solution of (E) with a prescribed boundary datum $\nu$. When $g(t)=t^q$ with $q \in (1,2)$, we show that equation (E) admits a critical exponent $q_\mu$ (depending only on $N$ and $\mu$). In the subcritical case, namely $1<q<q_\mu$, we establish some a priori estimates and provide a description of solutions with an isolated singularity on $\partial \Omega$. In the supercritical case, i.e. $q_\mu\leq q<2$, we demonstrate a removability result in terms of Bessel capacities.

	\medskip
	
	
	\medskip
	
	
\end{abstract}

\section{Introduction}

In this paper, we are concerned with the boundary value problem with measure data for the following equation
\bel{N} - \Gd u - \frac{\gm}{\gd^2}u + g(|\nabla u|) = 0 \ee
in a $C^2$ bounded domain $\Gw$ in $\BBR^N$ ($N \geq 3$), where  $\gm \in (0,\frac{1}{4}]$, $\gd(x)=\gd_\Gw(x):=\dist(x,\prt \Gw)$ and $g: \BBR_+ \to \BBR_+$ is a nondecreasing, continuous function with $g(0)=0$. Put
\begin{equation}\label{L}
  L_\mu=L_\mu^\Gw:=\Gd+\frac{\mu}{\gd^2}.
\end{equation}
We say that $u$ is $L_\mu$ harmonic (resp. subharmonic, superharmonic) if $u$ is a distributional solution (resp. subsolution, supersolution) of
\bel{LE} -L_\mu u = 0 \quad \text{in } \Gw. \ee

When $\mu=0$, $L_\mu \equiv \Gd$ and equation \eqref{N} becomes
\bel{mu=0} -\Gd u + g(|\nabla u|) = 0 \quad \text{in } \Gw.
\ee
The boundary value problem with measure data for \eqref{mu=0} was first studied by Nguyen and V\'eron in \cite{NgVe1} where the existence of a positive solution with a prescribed measure boundary datum was obtained under a so-called subcriticality integral condition on $g$. In case that $g$ is a purely power function, i.e. $g(t)=t^q$, it was shown that equation \eqref{mu=0} admits the \textit{critical exponent} $q^*=\frac{N+1}{N}$ and the structure of the class of solutions with a boundary isolated singularity was fully depicted in the subcritical case $q \in (1,q^*)$. These results were then extended to a much more intricate equations where the nonlinearity depends on both solutions and their gradient (see \cite{MaNg2,Ng1}). An attempt to extend the mentioned results was carried out in \cite{BGV} (see also the references therein) to the case of $p$-laplacian where the analysis is complicated and requires heavy computations due to the nonlinearity of the operator.

The case $\mu<0$ (and the case of more general potentials) was investigated by Ancona in \cite{An2}. 

When $\mu>0$, the semilinear equation with absorption power term
\bel{SLE} -L_\mu u + u^p = 0 \quad \text{in } \Gw \ee
has been recently studied in different directions. If $\mu \leq \frac{1}{4}$, Bandle et al. \cite{BMR} gave a classification of large solutions, i.e. positive solutions of \eqref{SLE} which blow up on $\prt \Gw$, according to their boundary behavior, in connection to the exponent
\bel{alpha} \ga:=\frac{1}{2} + \sqrt{\frac{1}{4}-\gm} \ee
and the Hardy constant
\bel{Hardy} C_H(\Gw):= \inf_{H^1_0(\Gw)} \frac{\int_\Gw |\nabla u|^2dx}{\int_\Gw (u/\gd)^2dx}.
\ee

Afterwards, Marcus and Nguyen dealt with \textit{moderate solutions}  of \eqref{LE} and \eqref{SLE} by introducing a concept of \textit{normalized boundary trace} (see \cite[Definition 1.2]{MaNg}). An advantage of this notion is that it allows to overcome the difficulty originating from the presence of the Hardy potential $\frac{\mu}{\gd^2}$ and hence enables to characterize $L_\mu$ harmonic functions in terms of their boundary behavior. 

This notion of boundary trace was then extended by Marcus and Moroz in \cite{MaMo} to the case $\mu<\frac{1}{4}$ due to the fact that in this case there exists a local $L_\mu$ superharmonic function in a neighborhood of $\prt \Gw$ and then it was used to study the nonlinear problem \eqref{SLE}. See also \cite{BMM} and references therein.

In parallel, Gkikas and V\'eron \cite{GkV} treated the boundary value problem for \eqref{LE} and \eqref{SLE} in a slightly different setting, giving a complete description of singular solutions of \eqref{SLE} by introducing a notion of boundary trace in a \textit{dynamic way} which is recalled below.

Let $D \Subset \Gw$ and $x_0\in D.$ If $h\in C(\prt D)$ then the following problem
\bel{sub} \left\{ \BAL
-L_{\xm }u &=0\qquad&&\text{in } D, \\
u &=h\qquad&&\text{on } \prt D,
\EAL \right. \ee
admits a unique solution which allows to define the $L_{\xm }$ harmonic measure $\xo_D^{x_0}$ on $\prt D$
by
\bel{redu2}
u(x_0)=\int_{\prt D}{}h(y)d\gw^{x_0}_{D}(y).
\ee
A sequence of domains $\{ \Gw_n \}$ is called a \textit{smooth exhaustion} of $\Gw$ if $\prt \Gw_n \in C^2$, $\overline{\xO_n}\subset \xO_{n+1}$, $\cup_n\xO_n=\xO$ and $\mathcal{H}^{N-1}(\prt \xO_n)\to \mathcal{H}^{N-1}(\prt \xO)$. For each $n$, let $\xo_{\Gw_n}^{x_0}$ be the $L_{\gm }^{\Gw_n}$ harmonic measure on $\prt\Gw_n.$

\begin{definition} \label{nomtrace}
	Let $ \mu \in (0,\frac{1}{4}]$. A function $u$ possesses a \emph{boundary trace}  if there exists a measure $\gn\in\GTM(\prt\Gw)$ such that for any smooth exhaustion $\{ \Gw_n \}$ of $\xO$, there  holds
	\be\label{trab}
	\lim_{n\rightarrow\infty}\int_{ \partial \xO_n}\gf u\, d\xo_{\Gw_n}^{x_0}=\int_{\partial \xO} \gf \,d\xn \quad\forall \gf \in C(\overline{\xO}).
	\ee
	The boundary trace of $u$ is denoted by $\tr(u)$.
\end{definition}
It was showed in \cite{GkNg} that when $\mu \in (0,C_H(\Gw))$ the notion of boundary trace in Definition \ref{nomtrace} coincides with the notion of normalized boundary trace introduced in \cite{MaNg}. Since we would like to deal with the whole range $(0,\frac{1}{4}]$, we will employ Definition \ref{nomtrace}. However, we need an additional condition as follows:
\bel{eig}
\gl_\mu:=\inf_{\vgf \in H_0^1(\Gw) \setminus \{ 0\} }\frac{\int_{\Gw}(|\nabla \vgf|^2 -\frac{\mu}{\gd^2}\vgf^2)dx  }{\int_{\Gw}\vgf^2 dx}>0.
\ee

\medskip

{\bf Throughout the present paper, we assume that $\mu \in (0,\frac{1}{4}]$ and \eqref{eig} holds}. Under this condition, the Representation Theorem (see \cite{MaNg,GkV}) is valid, which allows to develop a theory for linear equations (see \cite{GkNg}).

Related results for semilinear elliptic equations with Hardy potential and source term can be found in \cite{BHNV,Ng2}.

For $\phi\geq 0$, denote by $\GTM(\Gw,\phi)$ the space of Radon measures $\tau$ on $\Gw$ satisfying $\int_{\Gw}\phi\, d |\tau|<\infty$ and by $\GTM^+(\Gw,\gf)$ the positive cone of $\GTM(\Gw,\gf)$. Denote by $\GTM(\prt \Gw)$ the space of bounded Radon measures on $\prt \Gw$ and by $\GTM^+(\prt \Gw)$ the positive cone of $\GTM(\prt \Gw)$. Denote  $L^p_w(\Gw,\tau)$, $1 \leq p < \infty$, $\tau \in \GTM^+(\Gw)$, the weak $L^p$ space (or Marcinkiewicz space) with weight $d\gt$; see \cite{MVbook} for more details.

Let $G_\mu^\Gw$ and $K_\mu^\Gw$ be respectively the Green kernel and Martin kernel of $-L_\mu$ in $\Gw$ (see Section 2.2 for more details). The Green operator and Martin operator are defined as follows:
\bel{Gm} \BBG_\mu^\xO[\tau](x):=\int_{\Gw}G_\mu^\xO(x,y)d\tau(y) \forevery \tau \in \GTM(\Gw,\gd^{\ga}), \ee
\bel{Km} \BBK_\mu^\xO[\nu](x): = \int_{\prt \Gw}K_\mu^\xO(x,z)d\nu(z) \forevery \nu \in \GTM(\prt \Gw). \ee

These operators play an important role in the study of the boundary value problem for the linear equation
\begin{equation}\label{NHL} \left\{ \BAL
- L_\gm u&=\tau\qquad \text{in }\;\Gw,\\
\tr(u)&=\xn.
\EAL \right. \end{equation}

\begin{definition} \label{solLinear} Let $(\tau,\xn)\in \GTM(\Gw,\gd^{\ga})\times \GTM(\prt \Gw)$. We say that $u$ is a \textit{weak solution} of \eqref{NHL} if $u \in L^1(\Gw,\gd^\ga)$ and
	\bel{lweakform}
	- \int_{\xO}u L_{\xm }\zeta \, dx=\int_\xO \zeta \, d\tau - \int_{\Gw} \mathbb{K}_{\xm}^\xO[\xn]L_{\xm }\zeta \, dx
	\qquad\forall \zeta \in\mathbf{X}_\xm(\xO),
	\ee
	where the space of test function ${\bf X}_\mu(\Gw)$ is defined by
	\bel{Xmu} {\bf X}_\mu(\Gw):=\{ \zeta \in H_{loc}^1(\Gw): \gd^{-\ga} \zeta \in H^1(\Gw,\gd^{2\ga}), \, \gd^{-\ga}L_\mu \zeta \in L^\infty(\Gw)  \}.
	\ee
\end{definition}
The existence and uniqueness result for \eqref{NHL}, which was established in \cite{GkNg}, is an important ingredient in the investigation of the boundary value problem for \eqref{N} 
\bel{BP} \left\{ \BAL
-L_\mu u+g(|\nabla u|) &=0 \quad \text{in } \Gw, \\
\tr(u)&=\nu.
\EAL \right. \ee

We reveal that the presence of the Hardy potential $\frac{\mu}{\gd^2}$  in the linear part of the equation means that the problem cannot be handled via classical elliptic PDEs methods as the singularity of the potential at the boundary is too strong. Moreover, the presence of the gradient term, which leads to the lack of monotonicity property of the nonlinearity, makes the analysis much intricate. The interplay between the Hardy potential $\frac{\mu}{\gd^2}$ and the gradient term yields substantial new difficulties and requires new methods.

Before stating main results of the paper, let us give the definition of weak solutions of \eqref{BP}.

\begin{definition} \label{weaksol}  Let $\nu \in \GTM(\prt \Gw)$.  A function $u$ is called a  \emph{weak solution} of \eqref{BP} if $u\in L^1(\Gw,\gd^{\ga})$, $g (|\nabla u|) \in L^1(\Gw,\gd^{\ga})$ and
	\bel{intform}- \int_{\Gw} u L_\gm\zeta \, dx + \int_{\Gw} g (|\nabla u|) \zeta \, dx = - \int_{\Gw}\BBK_\gm[ \gn] L_\gm \zeta \, dx, \quad \forall \zeta \in{\mathbf X_\xm}(\Gw).
	\ee
\end{definition}

Crucial ingredients in the study of \eqref{BP} are estimates of $\nabla \BBG_\mu^\Gw$ and $\nabla \BBK_\mu^\Gw$, which are established in the next proposition. \medskip

\noindent \textbf{Proposition A.} {\it
(i) Let $\gth \in [0,\ga]$ and $\xg\in[0,\frac{\theta N}{N-1})$. Then there exists a positive constant $c=c(N,\mu,\gth,\gg,\Gw)$ such that
\bel{gradG3} \| \nabla \BBG_\mu^\Gw[|\gt|] \|_{L_w^{\frac{N+\xg}{N+\theta-1}}(\Gw,\gd^\xg)} \leq c \| \gt   \|_{\GTM(\Gw,\gd^\gth)} \quad \gt \in \GTM(\Gw,\gd^\gth),
\ee
where
\bel{Gmgr} \nabla\BBG_\mu^\xO[\tau](x):=\int_{\Gw}\nabla_xG_\mu^\xO(x,y)d\tau(y). \ee

(ii) Let $\xg\geq0$.
Then there exists a positive constant $c=c(N,\mu,\gg,\Gw)$ such that
\bel{gradP2} \| \nabla \BBK_\mu^\Gw[|\gn|] \|_{L_w^{\frac{N+\xg}{N+\xa-1}}(\Gw,\gd^\gg)} \leq c \| \gn   \|_{\GTM(\prt \Gw)} \quad \gn \in \GTM(\prt \Gw),
\ee
where
\bel{Kmgr} \nabla\BBK_\mu^\xO[\nu](x): = \int_{\prt \Gw}\nabla_xK_\mu^\xO(x,z)d\nu(z). \ee
}

A main feature of problem \eqref{BP} is that, in general, it is not solvable for any measure $\nu \in \GTM(\prt \Gw)$. This occurs only when $q$ is smaller than the \textit{critical exponent} given by
$$q_\mu:=\frac{N+\ga}{N+\ga-1}.$$

\noindent \textbf{Theorem B.} {\it
(Existence)\, Assume that $g: \BBR_+ \to \BBR_+$ is continuous, nondecreasing and satisfies
\begin{equation}\label{G1} \tag{g$_1$}
\int_{1}^{\infty}g(s)s^{-1-q_\mu}ds<\infty.
\end{equation}
Then for any $\xn\in\mathfrak M^+(\prt\Gw)$ problem \eqref{BP} admits a nonnegative weak solution
$u=u_{\gn}$. Moreover,
\bel{uGK} u + \BBG_\mu^\Gw[g(|\nabla u|)] = \BBK_\mu^\Gw[\nu]. \ee
}

Let us briefly discuss the idea of the proof. Because of the presence of the Hardy potential, we first construct a solution of \eqref{N} in a subdomain $D \subset \subset \Gw$ due to a combination of the idea in \cite{Ka} and the Schauder fixed point theorem. This result is used to obtain the existence of an approximate solution of the equation with truncated nonlinearity in the whole domain $\Gw$. Finally, we employ  Proposition A and Vitali convergence theorem in the limit process to derive the existence of a solution of \eqref{BP}.

A combination of \eqref{uGK} and Schr\"odinger theory (see \cite{GkV,MaNg}  and references therein) asserts that any weak solution of \eqref{BP} behaves like $\BBK_\mu[\nu]$ on $\prt \Gw$.  \medskip

\noindent \textbf{Proposition C.} {\it (Boundary behavior)
Let $\nu \in \GTM^+(\prt \Gw)$. If $u$ is a nonnegative weak solution of \eqref{BP} then
\bel{nontang} \lim_{x \to y}\frac{u(x)}{\BBK_\mu^\Gw[\nu](x)}= 1 \quad \text{non tangentially},\, \text{for } \nu\text{-a.e. } y \in \prt \Gw.
\ee
}

Following is the monotonicity result which clearly implies the uniqueness the solution of \eqref{BP}.

\noindent \textbf{Theorem D.} {\it (Monotonicity)
Assume  that $g: \BBR_+ \to \BBR_+$ is continuous, nondecreasing and satisfies
	\bel{G2} \tag{g$_2$} |g(t)-g(t')| \leq C(|t|^{q-1}+|t'|^{q-1})|t -t'| \quad \forall t, t' \geq 0,\ee
	for some $q \in (1,q_\mu)$ and
	\bel{G3} \tag{g$_3$}
	g(\xe t) \leq \xe g(t)\quad\forall t>0\quad\text{and }\;\forall \xe\in (0,1].
	\ee
Assume $\nu_i \in \GTM^+(\prt \Gw)$, $\nu_1 \leq \nu_2$ and $u_i$ be a nonnegative solution of \eqref{BP} with $\nu=\nu_i$, $i=1,2$. If $\nu_1 \leq \nu_2$ then $u_1 \leq u_2$ in $\Gw$.		
} \medskip

Note that the classical method can not be applied to our setting because of the lack of monotonicity stemming from the presence of the gradient term and the fact that a constant is not a solution of \eqref{BP}. To overcome the difficulties,
we develop a new method which is based on an estimates on the gradient of subsolutions, Kato's inequality and a comparison principle in a subdomain of $\Gw$.

When $g(t)=t^q$ with $q \in (1,q_\mu)$ then $g$ satisfies \eqref{G1}--\eqref{G3}. In this case, the class of solutions with isolated singularity has an interesting structure which is exploited below.

Put
$$W(x):=\left\{ \BAL &\gd(x)^{1-\ga} \qquad&& \text{if } \xm <\frac{1}{4}, \\
&\gd(x)^{\frac{1}{2}} |\ln \gd(x)| &&\text{if } \xm =\frac{1}{4}.
\EAL \right.
$$

Assume that $0 \in \prt \Gw$. The following proposition provides universal pointwise estimates on solutions with isolated singularity at $0$, as well as their gradient. The proof is obtained thanks to the barrier constructed in the Appendix and the scaling argument in \cite{MVbook}. \medskip

\noindent \textbf{Proposition E.} {\it (A priori estimates)
Assume $0\in\partial\xO$ and let $u$
be a positive solution of \eqref{N} in $\xO,$ with $g(t)=t^q,$ such that
\bel{lama} \lim_{x\in\xO,\;x\rightarrow\xi}\frac{u(x)}{W(x)}=0\qquad\forall \xi\in\partial\xO\setminus \{0\},
\ee
locally uniformly in $\partial\xO\setminus \{0\}$. Then there exists a constant $C=C(N,\mu,q,\Gw)$ such that,
\bel{3.4.24}
u(x)\leq C\gd(x)^{\ga}|x|^{-\frac{2-q}{q-1}-\ga}\qquad\forall x\in \xO,
\ee
\bel{3.4.24*}
|\nabla u(x)|\leq C\gd(x)^{\ga-1}|x|^{-\frac{2-q}{q-1}-\ga}\qquad\forall x\in \xO.
\ee		
}

In case that the boundary trace is a Dirac measure concentrated at $0$, a shaper estimates can be achieved, which is the content of the following theorem. \medskip

\noindent \textbf{Theorem F.} {\it (Weak singularity)
Assume $1< q < q_\mu$. Let $0 \in\prt\Gw$ and $k > 0$.  Let $u_{0,k}^\Gw$ be the solution of
\bel{Pc} \left\{ \BAL -L_\mu u + \abs{\nabla u}^q  &= 0 \quad \text{in } \Gw \\
\tr(u) &= k\gd_0,
\EAL \right. \ee
where $\gd_y$ is the Dirac measure concentrated at $0$. Then
\bel{BA.1} \lim_{x \to y}\frac{u_{0,k}^\Gw(x)}{K_\mu^\Gw(x,0)}=k. \ee
Furthermore the mapping $k\mapsto u_{0,k}^\Gw$ is increasing.
}

From Theorem F, it is natural to analyze the behavior of $\lim_{k \to \infty}u_{0,k}^\Gw$. This task consists of some intermediate steps. The first one is to consider a separable solution of \eqref{N} in the case $\Gw=\BBR_+^N$ and then to translate equation \eqref{N} to an equation on the upper hemisphere
$$S_+^{N-1}=S^{N-1} \cap \BBR_+^N=\left\{(\sin\gf\gs',\cos\gf):\gs'\in S^{N-2},\gf\in [0,\frac{\gp}{2})\right\}.
$$
The second one is to investigate the existence and uniqueness of the  corresponding problem on $S_+^{N-1}$; at this step the exact behavior of $\lim_{k \to \infty}u_{0,k}^{\BBR_+^N}$ can be derived. In the last step, the scaling argument is employed to obtain the behavior of $\lim_{k \to \infty}u_{0,k}^\Gw$. These steps are described in more details below.

We denote by $x=(r,\gs) \in \BBR_+ \ti S^{N-1}$, with $r=|x|$ and $\gs=r^{-1}x$, the spherical coordinates in $\BBR^N $ and we recall the following representation
$$\nabla u=u_r{\bf e}+\frac{1}{r}\nabla' u, \quad \Gd u=u_{rr}+\frac{N-1}{r}u_r+\frac{1}{r^2}\Gd' u
$$
where $\nabla' $ denotes the covariant derivative on $S^{N-1}$ identified with the tangential derivative and $\Gd'$ is the Laplace-Beltrami operator on $S^{N-1}$.

We look for  a particular solution of
\bel{P3} \left\{ \BAL
- L_\mu u + \abs{\nabla u}^q  &= 0 \quad &&\text{ in } \BBR^N_+ \\
u &= 0 &&\text{ on } \prt \BBR^N_+\sms\{0\}=\BBR^{N-1}\sms \{0\}
\EAL \right. \ee
under the separable form
\bel{Sep-sol} u(x)=u(r,\gs)=r^{-\frac{2-q}{q-1}}\gw(\gs) \qq (r,\gs) \in (0,\ity)\ti S_+^{N-1}.  \ee
It follows from a straightforward computation that $\gw$ satisfies
\bel{P4} \left\{ \BAL
-\CL_\mu \gw - \ell_{N,q}\gw + J(\gw,\nabla' \gw) &= 0 \q &&\text{in } S_+^{N-1} \\
\gw &= 0 &&\text{on } \prt S_+^{N-1}
\EAL \right. \ee
where
\bel{CLmu} \BAL \CL_{\xm }\gw &:= \Gd'\gw+\frac{\xm }{ ({\bf e}_N \cdot\gs)^2}w,\qquad
\ell_{N,q} :=\frac{2-q}{q-1}\Big(\frac{q}{q-1}-N\Big), \\
J(s,\xi)&:=\left( \Big(\frac{2-q}{q-1}\Big)^2 s^2 + |\xi|^2\right)^{\frac{q}{2}} \quad (s,\xi)\in\mathbb{R}\times\mathbb{R}^{N}.
\EAL \ee
where ${\bf e}_N$ is the unit vector pointing toward the North pole.
 
Let $\gk_\mu=\xa(N+\xa-2)$ be the first eigenvalue of $-\CL_\mu$ in $S_+^{N-1}$ and $\gf_\mu$ be the corresponding eigenfunction $\gf_\mu(\gs)=({\bf e}_N \cdot \gs)^\ga$ for $\gs \in S_+^{N-1}$. Denote
\bel{Y} {\bf Y}_\mu(S^{N-1}_+):=\{\gf\in H^1_{loc}(S^{N-1}):\xf_0^{-\xa}\phi\in H^1(S^{N-1}_+,\xf_0^{2\ga})\}. \ee

It is asserted below that $q_\mu$ is a critical exponent for the existence of a positive solution of \eqref{P4}. \medskip

\noindent \textbf{Theorem G.} {\it
	(i) If $q \geq q_\mu$ then there exists no nontrivial solution of \eqref{P4}.
	
	(ii) If $1<q<q_\xm$ then problem \eqref{P4} admits a unique positive solution $\xo\in{\bf Y}_\mu(S^{N-1}_+)$. Moreover,
	\bel{estsphere} \BAL
	\xo(\xs)&\leq \left(\frac{\ell_{N,q}-\gk_\xm}{\xa^q}\right)^{\frac{1}{q-1}}\xf_\xm(\xs) \quad\forall\xs\in S^{N-1}_+,\\
	|\nabla' \xo(\xs)|&\leq C\xf_\xm(\xs)^{\frac{\xa-1}{\xa}} \quad\forall\xs\in S^{N-1}_+.\\
	\EAL \ee
where $C=C(N,q,\mu)$.
}

Denote $u_{0,\infty}^\Gw:=\lim_{k \to \infty}u_{0,k}^\Gw$. A combination of Proposition E, Theorem F and Theorem G ensures that $u_{0,\infty}^\Gw$ is a solution of \eqref{N}. Moreover,  $u_{0,\infty}^\Gw$ possesses richer properties as stated in the following theorem. \medskip

\noindent \textbf{Theorem H.} {\it (Strong singularity)
	Assume  $0 \in \prt \Gw$ and $1< q < q_\mu$. Let $u_{0,\infty}^\Gw$ be defined as above. Then $u_{0,\infty}^\Gw$ is a solution of
	\bel{uinfty} \left\{ \BAL -L_\mu u + |\nabla u|^q &=0 \quad &&\text{in } \Gw, \\
	u &= 0 &&\text{on } \prt \Gw \sms \{0\}.
	\EAL\right.
	\ee
	There exists a constant $c=c(N,\mu,q,\Gw)>0$ such that
	\bel{uinf2side}
	c^{-1}\gd(x)^\ga |x|^{-\frac{2-q}{q-1}-\ga} \leq u_{0,\infty}^\Gw(x) \leq c \gd(x)^\ga |x|^{-\frac{2-q}{q-1}-\ga} \quad \forall x \in \Gw,
	\ee
	\bel{gradinf}
	|\nabla u_{0,\infty}^\Gw(x)| \leq c\gd(x)^{\ga-1}|x|^{-\frac{2-q}{q-1}-\ga} \quad \forall x \in \Gw.
	\ee
	Moreover
	\bel{uniq7}
	\lim_{\tiny\BA{c}\Gw \ni x\to 0\\
		\frac{x}{|x|}=\gs\in S^{N-1}_+
		\EA}|x|^{\frac{2-q}{q-1}}u_{0,\infty}^\Gw(x)=\gw(\gs),
	\ee
	locally uniformly on $S^{N-1}_+$, where $\gw$ is the unique solution of \eqref{P4}.	
} \medskip

We next consider the supercritical case, i.e. $q \geq q_\mu$. For any Borel  set $E \subset \BBR^{N-1}$, we denote by $C_{\ga,p}^{\BBR^{N-1}}(E)$ the Bessel capacity of $E$ associated to the Bessel space $L_{\ga,p}(\BBR^N)$ (see Section 6 for more details).

\begin{definition}
	Let $\xn\in\mathfrak M^+(\prt\Gw).$ We will say that $\xn$ is absolutely continuous with respect to the Bessel capacity $C^{\BBR^{N-1}}_{\frac{\xa+1}{q}-\xa,q'}$ where $q'=\frac{q}{q-1}$, if
	\begin{equation}\label{N3}\BA {lll}
	\forall E\subset\prt\Gw,\,E\text { Borel }, C^{\BBR^{N-1}}_{\frac{\xa+1}{q}-\xa,q'}(E)=0\Longrightarrow \xn(E)=0.
	\EA\end{equation}
\end{definition}

\noindent \textbf{Theorem I.} {\it (Absolute continuity)
	Assume $q_\xm\leq q<2$ and $\xn \in \GTM^+(\prt \Gw)$ such that the problem
	\bel{N1} \left\{ \BAL -L_\mu u + \abs{\nabla u}^q  &= 0 \quad \text{in } \Gw \\
	\tr(u) &= \nu.
	\EAL \right. \ee
has a solution. Then
	
	(i) If $q\neq\xa+1$ then $\xn$ is absolutely continuous with respect to $C^{\BBR^{N-1}}_{\frac{\xa+1}{q}-\xa,q'}.$
	
	(ii) If $q=\xa+1$ then for any $\xe\in(0,\min\{\xa+1,\frac{(N-1)\xa}{\xa+1}-(1-\xa)\}),$ $\xn$ is absolutely continuous with respect to $C^{\BBR^{N-1}}_{\xe+1-\xa,\frac{\xa+1}{\xa}}.$\label{theorima1}
}

\noindent \textbf{Theorem J.} {\it (Removability)
	Assume $q_\xm \leq q < 2.$ Let $K\subset\partial\xO$ be compact such that
	
	(i) $C^{\BBR^{N-1}}_{\frac{\xa+1}{q}-\xa,q'}(K)=0$ if $q\neq\xa+1$
	
	or
	
	(ii) $C^{\BBR^{N-1}}_{\xe+1-\xa,q'}(K)=0,$ for some $ \xe\in (0,\min\{\xa+1,\frac{(N-1)\xa}{\xa+1}-(1-\xa)\}),$ if $q=\xa+1.$
	
	\noindent Then any nonnegative solution $u \in C^2(\xO)\cap C(\overline{\xO}\setminus K)$ of
	\bel{problem}
	-L_\xm u+|\nabla u|^{q}=0\quad\text{in }\Gw.
	\ee
	 such that
	\bel{lamab} \lim_{x\in\xO,\;x\rightarrow\xi}\frac{u(x)}{W(x)}=0\qquad\forall \xi\in\partial\xO\setminus K,
	\ee
	is identically zero.
}

The paper is organized as follows. In Section 2, we recall main properties of the boundary trace  and some facts about linear equations. In Section 3, we  establish estimates of the gradient of solutions in weak $L^p$ spaces (see Proposition A). Section 4 is devoted to the proof of Theorem B, Proposition C and Theorem D. Moreover, in this section, we also provide  some estimates of solutions of \eqref{N}. In Section 5, we demonstrate Proposition E and Theorems F, G and H. In Section 6 we deal with the supercritical case and provide the proof of Theorems I and J. Finally, in Appendix we construct a barrier for solutions of \eqref{N} which serves to obtain Proposition E. \medskip

\noindent \textbf{Notation.} In what follows the notation $f \approx g$ means: there exists a positive constant $c$ such that $c^{-1}f < g < cf$ in the domain of the two functions or in a specified  subset of this domain. Of course, in the later case, the constant depends on the subset. 

For $\gb>0$, put
\bel{Gwbeta} \Gw_\gb:=\{x \in \Gw: \gd(x)<\gb \},\; D_\gb:=\{x \in \Gw: \gd(x)>\gb\}, \; \Gs_\gb:=\{x \in \Gw:
\gd(x)=\gb\}.\ee
\medskip

\noindent{\bf Acknowledgements.} The authors are grateful to Professor L. V\'eron  for his useful comments.

\section{The linear problem}
\subsection{Eigenvalue and eigenfunction} Throughout the paper we assume that $\xm\in(0,\frac{1}{4}]$ and \eqref{eig} holds.

We recall important facts of the eigenvalue $\gl_\mu$ of $-L_\mu$ and the associated eigenfunction $\vgf_\mu$ which can be found in \cite {FMT}.

If $0<\xm <\frac{1}{4}$ then the minimizer $\vgf_{\xm } \in H_0^1(\Gw)$ of \eqref{eig} exists and satisfies
\be\label{Lin1}
\vgf_{\mu }\approx \gd^\ga,
\ee
where $\ga$ is defined by \eqref{alpha}. \smallskip

 If $\xm =\frac{1}{4}$, there is no minimizer of \eqref{eig} in $H_0^1(\Gw)$, but there exists a nonnegative function $\vgf_{\frac{1}{4}}\in H_{loc}^1(\xO)$  such that
\be\label{Lin2}\vgf_{\frac{1}{4}}\approx \gd^\frac{1}{2},\ee
and satisfies
$$-L_{\frac{1}{4}}\vgf_{\frac{1}{4}}=\xl_\xm \vgf_{\frac{1}{4}}\qquad\text{in}\;\;\xO$$
in the sense of distributions. In addition,  $\gd^{-\frac{1}{2}}\vgf_{\frac{1}{4}} \in H^1_0(\Gw, \gd)$.

\subsection{Green kernel and Martin kernel}
Let $G_\mu^\Gw$ and $K_\mu^\Gw$ be respectively the Green kernel and Martin kernel of $-L_\mu$ in $\Gw$ (see \cite{MaNg,GkV}) for more details). We recall that
\bel{Gmu} G^\Gw_\gm(x,y) \approx
\min\left\{|x-y|^{2-N},\gd(x)^{\ga}\gd(y)^{\ga}|x-y|^{2
	-N-2\ga}\right\} \quad \forall x,y \in \Gw, x \neq y,
\ee
\bel{Kmu} K_\gm^\Gw(x,y) \approx \gd(x)^{\ga}|x-y|^{2 - N - 2\ga} \forevery x \in
\Gw, \, y \in \prt \Gw. \ee

Denote  $L^p_w(\Gw,\tau)$, $1 \leq p < \infty$, $\tau \in \GTM^+(\Gw)$, the weak $L^p$ space (or Marcinkiewicz space) with weight $\gt$; see \cite{MVbook} for more details. Notice that, for every $s>-1$,
	\be
	L_w^p(\Gw,\gd^s) \sbs L^{r}(\Gw,\gd^s) \forevery r \in [1,p).\label{weakin2}
	\ee
	Moreover for any $u \in L_w^p(\Gw,\gd^{s})$ ($s>-1$),
	\bel{ue} \int_{\{|u| \geq \gl\} }\gd^{s} dx \leq \gl^{-p}\norm{u}^p_{L_w^p(\Gw,\gd^s)} \quad \forall \gl>0.
	\ee

Let $\BBG_\mu^\Gw$ and $\BBK_\mu^\Gw$ be the Green operator and Martin operator of $-L_\mu$ in $\Gw$ which are given in \eqref{Gm}, \eqref{Km}.

We recall estimate of Green kernel and Martin kernel in weak $L^p$ spaces (see \cite{GkNg}).

\begin{proposition} \label{GP} \emph{(i)} Let $\gg \in(-\frac{\ga N  }{N+2\ga   -2},\frac{\ga   N}{N-2})$. There exists a constant $c=c(N,\mu,\gg,\Gw)$ such that
	\bel{estG1}
	\norm{\BBG_\gm^\Gw[\gt]}_{L_w^{\frac{N+\gg}{N+\ga -2}}(\Gw,\gd^\gg)} \leq
	c\norm{\gt}_{\GTM(\Gw,\gd^{\ga  })} \q \forall \gt \in \GTM(\Gw,\gd^{\ga}).
	\ee
	\emph{(ii)} Let $\gg>-1$. Then there exists a constant $c=c(N,\mu,\gg,\Gw)$ such that
	\bel{estK}
	\norm{\BBK_\mu^\Gw[\gn]}_{L_w^{\frac{N+\gg}{N+\ga -2}}(\Gw,\gd^\gg)} \leq
	c\norm{\gn}_{\GTM(\prt \Gw)} \q \forall \gn \in \GTM(\prt \Gw). \ee
\end{proposition}

\subsection{Boundary trace}
In this subsection we recall main properties of the  boundary trace in connection with of $L_\xm$ harmonic functions. It is worth emphasizing that the below results are valid for $\mu \in (0,\frac{1}{4}]$ (under the condition that the first eigenvalue $\gl_\mu$ of $-L_\mu$ is positive).

\begin{proposition} \label{PropA} \emph{(\cite{GkNg})}   (i) For any $\gt \in \GTM(\Gw,\gd^\ga), \,  \tr(\BBG_\mu^\Gw[\gt])=0$ and for any $\nu \in \GTM(\prt \Gw)$,   $\tr(\BBK_\mu^\Gw[\nu])=\nu$.
	
	(ii) Let $w$ be a nonnegative $L_\xm$ subharmonic function in $\Gw$. Then $w$ is dominated by an $L_\xm$ superharmonic function if and only if $w$ has a boundary trace $\xn\in\mathfrak{M}(\partial\xO)$. Moreover, if $w$ has a boundary trace then $L_\mu w \in \GTM^+(\Gw,\gd^\ga)$. In addition, if $\tr(w) = 0$ then $w = 0$.
	
	(iii) Let $u$ be a nonnegative $L_\xm$ superharmonic function. Then there exist
	$\xn \in \mathfrak{M}^+(\partial\xO)$ and $\tau\in\mathfrak{M}^+(\xO,\xd^\xa)$ such that
	\be\label{formGK}
	u=\mathbb{G}_{\xm}^\xO[\tau]+\mathbb{K}_{\xm}^\xO[\xn],
	\ee
	
	(iv) Let $(\tau,\xn)\in \GTM(\Gw,\gd^{\ga})\times \GTM(\prt \Gw)$. Then there exists a unique weak solution $u$ of \eqref{NHL}. The solution is given by \eqref{formGK}. Moreover, there exists $c=c(N,\mu,\Gw)$ such that
	\be\label{poi3}
	\|u\|_{L^1(\xO,\gd^\ga)}\leq c(\|\tau\|_{\GTM(\Gw,\gd^{\ga})}+ \|\nu\|_{\GTM(\prt \Gw)}).
	\ee
In addition, for any $\zeta \in \mathbf{X}_\xm(\xO)$, $\zeta \geq 0$,
\be\label{poi4}
-\int_{\Gw}{}|u|L_{\xm }\zeta \, dx\leq \int_{\Gw}{}\zeta \, \sign(u)\, d\tau-
\int_{\Gw}{}\mathbb{K}_{\xm}^\xO[|\xn|] L_{\xm }\zeta \, dx,
\ee
\be\label{poi5}
-\int_{\Gw}{}u_+L_{\xm }\zeta \, dx\leq \int_{\Gw}{}\zeta\, \sign_+(u)\,d\tau-
\int_{\Gw}{}\mathbb{K}_{\xm}^\xO[\nu_+]L_{\xm }\zeta \,dx.
\ee
\end{proposition}

\section{Estimates of the gradient of Green kernel and Martin kernels} We begin this section by recalling well-known geometric properties of a $C^2$ bounded domain $\xO$.

\begin{proposition} \label{geo} There exists a positive constant $\xb_0$ such that $\gd\in C^{2}(\overline{\xO}_{4\xb_0})$. Moreover, for any $x\in \xO_{4\xb_0}$ there exists a unique $\xi_x\in \partial\xO$ such that

a) $\gd(x)=|x-\xi_x|$ and ${\bf n}_{\xi_x} = - \nabla \gd(x)= -\frac{x-\xi_x}{|x-\xi_x|}.$

b) $x(s):=x+s\nabla \gd(x)\in \xO_{\xb_0}$  and $\gd(x(s))=|x(s)-\xi_x|=\gd(x)+s,$ for any $0<s<4\xb_0-\gd(x).$
\end{proposition}

\begin{lemma} \label{harmonic} Let $D$ be a $C^2$ bounded domain in $\BBR^N$. If $u$ is a nonnegative $L_\mu$ harmonic function in $D$ then
	\bel{gradu-u1} |\nabla u(x)| \leq C \frac{u(x)}{\dist(x,\prt D)} \forevery x \in D. \ee
\end{lemma}
\begin{proof} Take an arbitrary point $x_* \in D$ and put
	$$ d:=\frac{1}{2}\dist(x_*, \prt D), \quad y_*:=\frac{1}{d}x_*,  \quad u_*(y):=u(d y), \quad
	y \in \frac{1}{d}D.$$
	Note that if $x \in B_{d}(x_*)$ then
	$y=\frac{1}{d}x \in B_1(y_*)$ and $1 \leq \dist(y,\prt (\frac{1}{d}D)) \leq 3$. In
	$B_1(y_*)$,
	$$ - \Gd u_* - \frac{\gm}{\dist(\cdot,\prt (\frac{1}{d}D))^2}u_* =
	d_*^{2}\Big(- \Gd u - \frac{\gm}{\gd^2}u\Big)=0. $$
	By local estimate for elliptic equations \cite[Theorem 8.32]{GT} and the Harnack inequality \cite[Theorem 8.20]{GT}, there exist positive constants $c_i=c_i(N,\gm)$, $i=1,2$ such that
	$$ \max_{B_{\frac{1}{2}}(y_*)}\abs{\nabla u_*} \leq
	c_1\max_{B_1(y_*)}u_* \leq c_2 \min_{B_1(y_*)}u_* . $$
	In particular,
	$$d_*\abs{\nabla u (x_*)} \leq c_2 u(x_*) $$
	which implies \eqref{gradu-u1}.
\end{proof}

Let us recall a result from \cite{BVi} which will be useful in the sequel.

\begin{proposition} \label{bvivier} \emph{(\cite[Lemma 2.4]{BVi})}
	Let $\gw$ be a nonnegative bounded Radon measure in $D=\Gw$ or $\partial \Gw$ and $\eta\in C(\xO)$ be a positive weight function. Let $H$ be a continuous nonnegative function
	on $\{(x,y)\in\xO\times D:\;x\neq y\}.$ For any $\xl > 0$ we set
	$$
	A_\xl(y):=\{x\in\xO\setminus\{y\}:\;\; H(x,y)>\xl\}\quad \text{and} \quad
	m_{\xl}(y):=\int_{A_\xl(y)}\eta(x)dx.
	$$
	Suppose that there exist $C>0$ and $k>1$ such that $m_{\xl}(y)\leq C\xl^{-k}$ for every $\gl>0$.  Then the operator
	$$\BBH[\gw](x):=\int_{D}H(x,y)d\gw(y)$$
	belongs to $L^k_w(\Gw,\eta )$ and
	$$\left|\left|\BBH[\gw]\right|\right|_{L^k_w(\Gw,\eta)}\leq (1+\frac{Ck}{k-1})\gw(D).$$
\end{proposition}

\begin{lemma} \label{gradGK2} Let $\gth \in [0,\ga]$ and $\xg\in[0,\frac{\theta N}{N-1})$. Then there exists a positive constant $c=c(N,\mu,\gth,\gg,\Gw)$ such that \eqref{gradG3} holds.
\end{lemma}
\begin{proof} By Lemma \ref{harmonic}, there exists $C=C(\xm,N,\xO)>0$ such that
	\be
	|\nabla_x G_\mu^\Gw(x,y)| \leq C \frac{G_\mu^\Gw(x,y)}{\min\left(\xd(x),|x-y|\right)} \quad\forall\;x,y\in\xO,\;x\neq y.\label{1}
	\ee
	Set
	\begin{align*}
	A_\xl(y)&:=\Big\{x\in\xO\setminus\{y\}:\;\; \frac{G_\mu^\Gw(x,y)}{\xd(y)^\gth\min\left(\xd(x),|x-y|\right)}>\xl \Big \},\\
	A_{\xl,1}(y)&:=\Big \{x\in\xO\setminus\{y\}:\;\; \xd(x)\leq|x-y| \text{ and } \frac{G_\mu^\Gw(x,y)}{\xd(y)^\gth\xd(x)}>\xl \Big \},\\
	A_{\xl,2}(y)&:=\Big \{x\in\xO\setminus\{y\}:\;\; \frac{G_\mu^\Gw(x,y)}{\xd(y)^\gth|x-y|}>\xl \Big \}\\
	m_{\xl}(y)&:=\int_{A_\xl(y)}\xd(x)^\gg dx,\quad m_{\xl,i}(y):=\int_{A_{\xl,i}(y)}\xd(x)^\gg dx, \quad i=1,2.
	\end{align*}
	Then it is easy to see that
	$$A_\xl(y)\subset A_{\xl,1}(y)\cup A_{\xl,2}(y) \quad\forall y\in\xO,$$
	which implies
	\be
	m_{\xl}(y)\leq  m_{\xl,1}(y)+m_{\xl,2}(y) \quad\forall y\in\xO.\label{mainstep}
	\ee
	
	Since $\xO$ is $C^2$ there exists $\xb_0>0$ such that for any $ x\in\xO_{3\xb_0}$ there exists a unique $\xi\in \partial\xO$  satisfies $|x-\xi|=\gd(x).$ Furthermore there exists a $C^2$ function $\xG:\mathbb{R}^{N-1}\rightarrow\mathbb{R}$ such that (upon relabeling and reorienting the coordinate axes if necessary) we have
	$$\xO\cap B(\xi,2\xb_0)=\{x=(x',x_N) \in B(\xi,2\xb_0): x_N>\xG(x')\}.$$
	
	\noindent \textbf{Step 1.}
	\emph{We will show that there exists $C=C(N,\mu,\gth,\gg,\Gw)$ such that
		\be
		m_{\xl,1}(y)\leq C \xl^{-\frac{N+\xg}{N+\theta-1}} \quad \forall y\in\Gw_{\gb_0}. \label{step1}
		\ee}
	To prove that, we note, for any $x,\;y\in\xO$, $x\neq y$, that
	\begin{align}\nonumber
	&\min\left\{1,\gd(x)^{\ga}\gd(y)^{\ga}\abs{x-y}^{-2\ga}\right\}\\ \nonumber
	&=\left(\min\left\{1,\gd(x)^{\ga}\gd(y)^{\ga}\abs{x-y}^{-2\ga}\right\}\right)^{1-\frac{\theta}{\xa}}
	\left(\min\left\{1,\gd(x)^{\ga}\gd(y)^{\ga}\abs{x-y}^{-2\ga}\right\}\right)^{\frac{\theta}{\xa}}\\ \nonumber
	&\leq C\left(\frac{\xd(x)}{|x-y|}\right)^{\xa(1-\frac{\theta}{\xa})}\left(\frac{\xd(x)\xd(y)}{|x-y|^2}\right)^{\xa\frac{\theta}{\xa}}\\ \nonumber
	&\leq C\frac{\xd(x)^\ga}{|x-y|^\xa}\frac{\xd(y)^\gth}{|x-y|^\theta}.
	\end{align}
	
	By \eqref{Gmu} and the above inequality we obtain
	\be
	\frac{G_\mu^\Gw(x,y)}{\xd(y)^\gth \gd(x)}\leq \frac{C}{\xd(x)^{1-\xa}|x-y|^{N+\xa+\theta-2}},\label{stathera}
	\ee
	where $C=C(N,\mu,\Gw)$, which implies $
	A_{\xl,1}(y)\subset \widetilde{A}_{\xl,1}(y)$  for every  $y\in\xO$
	where
	\be
	\widetilde{A}_{\xl,1}(y):=\Big \{x\in\xO\setminus\{y\}:\;\; \xd(x)\leq|x-y| \text{ and } \frac{C}{\xd(x)^{1-\xa}|x-y|^{N+\xa+\theta-2}}>\xl \Big \}. \label{const}
	\ee
	Set
	$$
	\widetilde{m}_{\xl,1}(y):=\int_{\widetilde{A}_{\xl,1}(y)}\xd(x)^{\xg}dx.
	$$
	Without loss of the generality we assume that $0\in\partial\xO,$ $\xd(y)=|y|$ and there exists a $C^2$ function $\xG:\mathbb{R}^{N-1}\rightarrow\mathbb{R}$ such that $\xG(0)=0,\;\nabla \xG(0)=0$ and
	$$\xO\cap B(0,2\xb_0)=\{x=(x',x_N)\in B(0,2\xb_0): x_N>\xG(x')\}.$$
	Let $\xl>C$  where $C$ is the constant in \eqref{const}.
	
	Set $Q_{\xl}:={\xl}^{\frac{1}{N+\theta-1}}\Gw$ and let $\xd_{Q_\xl}$ be the distance function to  $\partial Q_{\xl}$. For any $\psi \in Q_{\xl}$, there exists $x\in\xO$ such that $\psi={\xl}^{\frac{1}{N+\theta-1}} x.$ In addition we have $\xd_{Q_\xl}(\psi)={\xl}^{\frac{1}{N+\theta-1}} \xd(x)$ and
	\be
	Q_{\xl}\cap B(0,2\xb_0{\xl}^{\frac{1}{N+\theta-1}})=\{\psi=(\psi',\psi_N) \in B(0,2\xb_0{\xl}^{\frac{1}{N+\theta-1}}):\psi_N>{\xl}^{\frac{1}{N+\theta-1}} \xG( {\xl}^{-\frac{1}{N+\theta-1}} \psi')\}.\label{allagi}
	\ee
	Set $\psi={\xl}^{\frac{1}{N+\theta-1}}x$ and $y_{\xl}={\xl}^{\frac{1}{N+\theta-1}}y,$ then  by change of variables we have
	\begin{align}
	\int_{\widetilde{A}_{\xl,1}(y)\cap B(0,2\xb_0)} \xd(x)^{\xg}dx ={\xl}^{-\frac{N+\xg}{N+\theta-1}}\int_{O\cap B(0,2\xb_0{\xl}^{\frac{1}{N+\theta-1}})} \xd_{Q_\xl}(\psi)^{\xg}d\psi,
	\end{align}
	where $$O:= \Big \{\psi\in Q_{\xl} \setminus\{y_{\gl}  \}:\;\;  \gd_{Q_\gl}(\psi) \leq |\psi-y_\gl| \text{ and } \frac{C}{\xd_{Q_\gl}(\psi)^{1-\xa}|\psi-y_{\xl}|^{N+\xa+\theta-2}}>1 \Big \}.$$
	
	Set $d_{\xl}(\psi)=\psi_N-{\xl}^{\frac{1}{N+\theta-1}}\xG({\xl}^{-\frac{1}{N+\theta-1}}\psi'),$  it can be checked that
	\bel{q1} \frac{1}{1+c_0}d_{\xl}(\psi)\leq \xd_{Q_\xl}(\psi)\leq d_{\xl}(\psi) \quad\forall\psi\in Q_{\xl}\cap B(0,2\gb_0 {\xl}^{\frac{1}{N+\theta-1}}), \ee
	where $c_0=\sup_{|x'|\leq M}|\nabla \xG(x')|$ with $M=\sup_{(x',x_N)\in\partial\xO}|x'|.$
	
	Set $z=(z',z_N)$ with $z_N=d_{\xl}(\psi)$ and $z'=\psi'$. By change of variables and by the above arguments we obtain
	\bel{lum}
	\int_{O\cap B(0,2\xb_0{\xl}^{\frac{1}{N+\theta-1}})} \xd_{Q_\xl}(\psi)^{\xg}d\psi \leq\int_0^{ c_1} z_N^\xg\int_{\{|z'-y_{\xl}'|\leq c_2(z_N)^{\frac{\xa-1}{N+\xa+\theta-2}}\}} dz'dz_N \leq c_3,
	\ee
	where $c_i=c_i(N,\mu,\gth,\Gw)$, $i=1,2$ and $c_3(N,\mu,\gth,\gg,\Gw)$.
	
	Combining the above estimates leads to
	\begin{align}
	\int_{A_{\xl,1}(y)\cap B(0,2\xb_0)} \xd(x)^{\xg}dx \leq c{\xl}^{-\frac{N+\xg}{N+\theta-1}},\quad\forall\xl>C, \label{in1}
	\end{align}
	where $c=c(\xO,\xm,N,\theta,\xg)$ and $C$ is the constant in \eqref{const}.
	
	Next, we estimate
	\bel{in2} \BAL
	\int_{A_{\xl,1}(y)\setminus B(0,2\xb_0)} \xd(x)^{\xg}dx &\leq \int_{ \big \{x\in\xO:\; \gd(x)^{1-\xa}\leq\frac{C}{\xl\xb_0^{N+\xa+\theta-2}} \big \}}\xd(x)^{\xg}dx\\
	&\leq\left(\frac{C}{\xl\xb_0^{N+\xa+\theta-2}}\right)^{\frac{\xg}{1-\xa}}\int_{ \big \{x\in\xO:\; \gd(x)^{1-\xa}\leq\frac{C}{\xl\xb_0^{N+\xa+\theta-2}} \big \}}dx\\
	& \leq c'\xl^{-\frac{\xg+1}{1-\xa}} \leq  c''{\xl}^{-\frac{N+\xg}{N+\theta-1}},
	\EAL \ee
	where $c'$ and $c''$ depend on $N,\mu,\gth,\gg,\Gw$.
	
	Thus \eqref{step1} follows by \eqref{in1} and \eqref{in2}. \medskip

	\noindent\textbf{Step 2.}\emph{We will show that there exists a constant $C=C(N,\mu,\gth,\gg,\Gw)$ such that
		\bel{step2}
		m_{\xl,1}(y)\leq C \xl^{-\frac{N+\xg}{N+\theta-1}}\quad \forall y\in {\overline D_{\gb_0}},
		\ee}
	where $D_{\gb_0}$ is defined in \eqref{Gwbeta}.
	
	Indeed, from \eqref{stathera} it is easy to see that
	\bel{ipo1}
	A_{\xl,1}(y)\cap \Gw_{\frac{\gb_0}{2}} \subset \Big \{x\in\xO:\; \xd(x)\leq|x-y| \text{ and } \xd(x)^{1-\xa}\leq\frac{C}{\xl
		\left(\frac{\xb_0}{2}\right)^{N+\xa+\theta-2}} \Big \}.
	\ee
	From \eqref{Gmu}, there exists $C=C(N,\mu,\Gw)$ such that, for every $(x,y)\in\xO\times\xO,\;x\neq y$,
	\bel{2.33}
	G_\xm^\Gw(x,y)\leq C \xd(y)^{\theta  }|x-y|^{2-N-\theta   },
	\ee
	which implies
	\bel{Glen}
	\frac{G_\mu^\Gw(x,y)}{\gd(y)^\theta\xd(x)}\leq \frac{C}{\xd(x)|x-y|^{N+\theta-2}}.
	\ee
	This leads to
	\bel{ipo2}
	A_{\xl,1}(y)\cap {\overline D_{\frac{\gb_0}{2}}} \subset \Big \{x\in\xO:\; \xd(x)\leq|x-y| \text{ and } |x-y|^{N+\theta-2}\leq\frac{C}{\xl\frac{\xb_0}{2}} \Big \},
	\ee
	where $C$ is the constant in \eqref{Glen}.
	
	Thus by \eqref{ipo1} and \eqref{ipo2}, for every $\gl \geq 1$, we have
	\begin{align*}
	\int_{A_{\xl,1}(y)} \xd(x)^{\xg}dx \leq C\left({\xl}^{-\frac{N+\xg}{N+\theta-2}}+\xl^{-\frac{\xg+1}{1-\xa}}\right) \leq  C{\xl}^{-\frac{N+\xg}{N+\theta-1}},
	\end{align*}
	which yields \eqref{step2}. \medskip
	
	\noindent \textbf{Step 3.} \emph{We will show that there exists a constant $C=C(N,\mu,\gth,\gg,\Gw)$ such that
		\be
		m_{\xl,2}(y)\leq C \xl^{-\frac{N+\xg}{N+\theta-1}} \quad \forall y\in\xO.\label{step3}
		\ee}
	
	From \eqref{Gmu}, there exists $C=C(N,\mu,\Gw)$ such that, for every $(x,y)\in\xO\times\xO,\;x\neq y$,
	\bel{2.34}
	G_\xm^\Gw(x,y) \leq C \frac{\xd(y)^{\theta   }}{\xd(x)^{\theta   }}|x-y|^{2-N}.
	\ee
	For any $x\in A_{\xl,2}(y)$, by \eqref{2.33} and \eqref{2.34},
	\bel{2.36}
	\xl\leq C|x-y|^{1-N-\theta   } \quad \text{and} \quad \gd(x)^{\gth   }\leq \frac{C}{\xl}|x-y|^{1-N}.
	\ee
	This ensures
	\begin{align*}
	m_{\xl,2}(y)=\int_{A_{\xl,2}(y)}\gd(x)^{\xg}dx\leq \int_{A_{\xl,2}(y)}\left(\frac{C}{\xl}|x-y|^{1-N}\right)^{\frac{\xg}{\theta   }}dx \leq C \xl^{-\frac{N+\xg}{N+\gth -1}},
	\end{align*}
	where in the above inequality we have used the fact that $\xg<\frac{\gth N}{N-1}$. \medskip
	
	\noindent \textbf{Step 4.} \emph{End of proof.}
	
	We infer from \eqref{step1}, \eqref{step2}, \eqref{step3} and \eqref{mainstep} that
	\bel{mgl>1} m_\gl(y) \leq C\gl^{-\frac{N+\gg}{N+\gth-1}}, \quad \forall \gl >C(\Gw,N,\mu)>1. \ee
	From that we can deduce that \eqref{mgl>1} also holds for every $\gl>0$. Therefore by applying Proposition \ref{bvivier} with $D=\xO,$ $\eta=\gd^{\xg}$ with $\xg\geq0$, $d\gw=\xd^{\theta  }d\tau$
	and
	$$H(x,y)=\frac{G_\xm^\Gw(x,y)}{\min\left(\xd(x),|x-y|\right)\xd(y)^\theta},$$
	we obtain
	$$ \| \nabla \BBH[|\gw|] \|_{L_w^{\frac{N+\xg}{N+\theta-1}}(\Gw,\gd^\xg)} \leq c \| \tau   \|_{\GTM(\Gw,\xd^\theta)}.
	$$
	Thus the result follows by \eqref{1}.
\end{proof}

\begin{lemma} \label{gradP} Let $\xg\geq0$.
	Then there exists a positive constant $c=c(N,\mu,\gg,\Gw)$ such that \eqref{gradP2} holds.
\end{lemma}
\begin{proof}  We use a similar argument as in the proof of Lemma \ref{gradGK2}. By Lemma \ref{harmonic}, there exists $C=C(N,\mu,\xO)>0$ such that
	\bel{2}
	|\nabla_x K_\mu^\Gw(x,y)| \leq \frac{C}{\xd(x)^{1-\xa}|x-y|^{N+2\xa-2}} \quad\forall\;x\in\xO,\, y\in\partial\xO.
	\ee
A similar, and simpler, argument as in the proof of Lemma \ref{gradGK2} justifies \eqref{gradP2} and hence we omit the proof.
\end{proof}

\noindent \textbf{Proof of Proposition A.} The proposition follows from Lemma \ref{gradGK2} and Lemma \ref{gradP}. \qed

\section{Subcritical absorption}
\subsection{Existence} Let $g: \BBR \to \BBR$ be a locally
Lipschitz continuous nonnegative and nondecreasing function vanishing at $0$. In this subsection, we deal with the existence of a solution of \eqref{BP}.

\begin{proof}[{\bf Proof of Theorem B}] \textbf{Step 1:} First we assume that $\sup_{t\in\mathbb{R}}|g(t)|=M<\infty.$ Let $D\subset\subset\xO$ be a smooth open domain and consider the equation
\bel{sub1}
-L_{\xm }v+g\left(\left|\nabla v+\nabla\BBK_{\xm}^\xO[\xn]\right|\right) =0\quad\text{in } D
\ee
First we note that $u_1=0$ is supersolution of \eqref{sub1} and $u_2=-\BBK_{\xm}^\xO[\xn]$ is a solution of \eqref{sub1}. Let
\bel{Tu} T(u):=\left\{ \BAL &0,\quad&&\text{if}\;\;0\leq u,\\
     &u\quad&&\text{if}\;\;u_2\leq u\leq0,\\
      &u_2\quad&&\text{if}\;\;u\leq u_2.
\EAL \right. \ee

In this step we use the idea in \cite{Ka} in order to construct a solution $v\in W^{1,\infty}(D)$ of the following problem
\bel{BP2} \left\{ \BAL
-L_{\xm }v+g\left(\left|\nabla v+\nabla\BBK_{\xm}^\xO[\xn]\right|\right)&=0 \quad &&\text{in } D, \\
v&=0 \quad&&\text{on } \prt D,
\EAL \right. \ee
which satisfies
\be
-\BBK_{\xm}^\xO[\xn]\leq v\leq0\quad\forall x\in D.\label{mainest}
\ee

Let $d_{D,\Gw}:=\dist(\partial D,\partial\xO)$ and $u\in W^{1,1}(D)$. By the standard elliptic theory, there exists a unique solution of the problem
\bel{BP3} \left\{ \BAL
-\xD w+(d_{D,\Gw}^{-2}-\xm \gd^{-2})w&=-g\left(\left|\nabla u+\nabla\BBK_{\xm}^\xO[\xn]\right|\right)+d_{D,\Gw}^{-2}T(u) \quad &&\text{in } D, \\
w&=0 \quad&&\text{on } \prt D.
\EAL \right. \ee
Recall that $\gd=\dist(\cdot,\prt \Gw)$.

We define an operator $\BBA$ as follows: to each $u \in W^{1,1}(D)$, we associate   the unique solution $\BBA[u]$ of \eqref{BP3}. Furthermore, since
$$d_{D,\Gw}^{-2}-\xm \gd(x)^{-2} \geq (1-\xm)\gd(x)^{-2} \quad \forall x \in D,$$
by the standard elliptic estimates we can obtain the existence of a positive constant $C=C(N,\mu,d_{D,\Gw},D)$ such that
\be\label{sup1}
\sup_{x\in D}|\BBA[u](x)|\leq C(M+\| \xn   \|_{\GTM(\partial\Gw)})=:C_1.
\ee
Also, by \eqref{sup1} and standard elliptic estimates, we have that there exists a positive constant $C=C(N,\mu,d_{D,\Gw},D)$ such that
\be\label{sup2}
\sup_{x\in D}|\nabla \BBA[u](x)|\leq C(M+\| \xn   \|_{\GTM(\partial\Gw)})=:C_2.
\ee

We will use the fixed point theorem to prove the existence of a fixed point of $\BBA$ by examining the following criteria.

\textit{We claim that $\BBA$ is continuous}. Indeed, if $u_n\rightarrow u$ in $W^{1,1}(D)$ then since $g \in C(\BBR_+) \cap L^\infty(\BBR_+)$, it follows that $g(|\nabla u_n + \nabla \BBK_\mu^\Gw[\nu]|) \to g(|\nabla u + \nabla \BBK_\mu^\Gw[\nu]|)$ and $T[u_n]\to T[u]$ in $W^{1,1}(D)$. Hence $\BBA[u_n]\to \BBA[u]$ in $W^{1,1}(D )$.

\textit{Next we claim that $\BBA$ is compact}. Indeed, let $\{u_n\}$ be a sequence in $W^{1,1}(D)$ then by \eqref{sup1} and \eqref{sup2}, $\{\BBA[u_n]\}$ is uniformly bounded in $W^{1,\infty}(D)$. Therefore there exists $\psi\in W^{1,p}_{loc}(D) $ and a subsequence still denoted by $\{\BBA[u_n]\}$ such that $\BBA[u_n]\rightarrow \psi$ in $L^p_{loc}(D)$ and
$\nabla \BBA[u_n]\to \nabla \psi $ weakly in $L^p_{loc}(D)$ and a.e. in $D$. By dominated convergence theorem we deduce that $\BBA[u_n]\rightarrow \psi$ in $W^{1,1}(D)$.

Now set
$$ \CK:=\{ \xi \in W^{1,1}(D):  \|  \xi \|_{W^{1,\infty}(D)} \leq C_1+C_2  \}. $$
Then $\CK$ is a closed, convex subset of $W^{1,1}(D)$ and $\BBA(\CK) \sbs \CK$.
Thus we can apply Schauder fixed point theorem to obtain the existence of a function $v \in \CK$ such that $\BBA[v]=v$. This means $v$ is a weak solution of \eqref{BP3}.

By the standard elliptic theory, we can easily deduce that $v,u_2\in C^2(D)\cap C(\overline{D})$. Moreover, it can be seen that $v\leq 0$.

Now we allege that $v\geq u_2$  by employing an argument of contradiction. Suppose $x_0\in D$ is such that
$$\inf_{x\in D}(v(x)-u_2(x))=v(x_0)-u_2(x_0)<0.$$
Then $\nabla v(x_0)=\nabla u_2(x_0)$, $-\xD (v-u_2)(x_0) \leq 0$ and $T[v](x_0)=T[u_2](x_0)=u_2(x_0)$.
But
\begin{align*}
-\xD (v-u_2)(x_0)=-(d_{D,\Gw}^{-2}-\xm \gd(x_0)^{-2})(v(x_0)-u_2(x_0))>0,
\end{align*}
which is clearly a contradiction.

As a consequence, $T(v)=v$ and therefore $v$ is a solution of of \eqref{BP2}.
\medskip

\noindent \textbf{Step 2:} Let $\{\Gw_n\}$ be a smooth exhaustion of $\Gw$ and let $v_n$ be the solution of \eqref{BP2} in $D=\xO_n$ satisfying \eqref{mainest}.
Then
\bel{est}
|v_n(x)|
\leq \BBG_\mu^{\Gw}\left[\chi_{\xO_n}g\left(\left|\nabla v_n+\nabla\BBK_{\xm}^\xO[\xn]\right|\right)\right](x) \leq CM\gd(x)^\xa \quad\forall x\in\xO_n,
\ee
for some positive constant $C=C(N,\xm,\Gw)$, where $M=\sup_{t \in \BBR}|g(t)|$. This  implies that there exists a subsequence, still denoted by $\{v_n\},$ such that $v_n\rightarrow v$ in $W^{1,p}_{loc}(\xO)$ and $v$ satisfies
\bel{BP4} \left\{ \BAL
-L_{\xm }v+g\left(\left|\nabla v+\nabla\BBK_{\xm}^\xO[\xn]\right|\right)&=0 \quad \text{in } \xO, \\
\tr(v)&=0 .
\EAL \right. \ee
Furthermore,
\be
-\BBK_{\xm}^\xO[\xn]\leq v\leq0 \quad\forall x\in \Gw.\label{mainest1}
\ee
Setting $u=v+\BBK_{\xm}^\xO[\xn],$ then $u$ is a solution of  \eqref{BP} satisfying $0 \leq u \leq \BBK_\mu^\Gw[\nu]$ in $\Gw$. \medskip

\noindent \textbf{Step 3:} Set $g_n:=\min(g,n)$ and let $u_n$ be a nonnegative solution of
\be\nonumber\left\{ \BAL
-L_\mu u_n+g_n(|\nabla u_n|)&=0 \quad\text{in } \Gw, \\
\tr(u_n)&=\nu ,
\EAL \right. \ee
satisfying
\bel{bound11} 0 \leq u_n \leq \BBK_\mu^\Gw[\nu] \quad \text{in } \Gw. \ee
Then $u_n$ satisfies
\bel{M99}- \int_{\Gw}{}u_n L_\gm\zeta \, dx + \int_{\Gw}{} g_n (|\nabla u_n|) \zeta \, dx = - \int_{\Gw}{}\BBK_\gm[ \gn] L_\gm \zeta \, dx \quad \forall \zeta \in{\mathbf X_\xm}(\Gw),
\ee
\be
u_n +\BBG_{\xm}^\xO[g_n(|\nabla u_n|)]=\BBK_{\xm}^\xO[\nu]. \label{repr}
\ee

Choosing $\zeta=\vgf_\mu$, we have by \eqref{poi4}
\be\label{poi4bb}
\xl_\xm\int_{\Gw}{}|u_n|\vgf_\mu dx+ \int_{\Gw}{}g_n(|\nabla u_n|)\vgf_\mu dx\leq
\xl_\xm\int_{\Gw}{}\mathbb{K}_{\xm}^\xO[|\xn|]\vgf_\mu dx.
\ee
Now by \eqref{repr}, Lemma \ref{gradGK2} and Lemma \ref{gradP} we obtain
\bel{gradG3b} \| \nabla u_n \|_{L_w^{q_\mu}(\Gw,\gd^\xa)} \leq c(N,\xm,\Gw) (\| g_n(|\nabla u_n|)   \|_{L^1(\Gw,\gd^\xa)}+\| \xn   \|_{\GTM(\partial\Gw)}).
\ee
Thus by \eqref{poi4bb} we have
\bel{gradG3bb} \| \nabla u_n \|_{L_w^{q_\mu}(\Gw,\gd^\xa)} \leq c(N,\xm,\Gw) \| \xn   \|_{\GTM(\partial\Gw)},
\ee
Similarly we can show that
\bel{G3bb} \|  u_n \|_{L_w^{q_\mu}(\Gw,\gd^\xa)} \leq c(N,\xm,\Gw) \| \xn   \|_{\GTM(\partial\Gw)}.
\ee
By \ref{weakin2}, for any $1<p<q_\mu$, $\{u_n\}$ is uniformly bounded in $W^{1,p}(\xO,\gd^\ga)$. Thus there exist $u\in W^{1,p}_{loc}(\xO) $ and a subsequence still denoted by $\{u_n\}$ such that $u_n\to u$ a.e. in $\Gw$ and
$\nabla u_n\to \nabla u $ a.e. in $\Gw$. Then \eqref{bound11} and the dominated convergence theorem guarantees that $u_n\rightarrow u$ in $L^1(\xO,\xd^\xa)$.

For $s>0$, set $E_n(s)=\{x\in \Gw:|\nabla u_n(x)|>s\}$. Then by \eqref{ue} and \eqref{G3bb},
$$e_n(s):=\int_{E_n(s)}\gd^\ga dx\leq s^{-q_\mu}\left(c(\xO,N,\xm) \| \xn   \|_{\GTM(\partial\Gw)}\right)^{q_\mu}.$$
Let $G\subset\Gw$ be a Borel subset. Then for any $s_0>0$
\begin{align*}
\int_{G} |g_n(|\nabla u_n|)|\gd^\ga dx&\leq \int_{G} |g(|\nabla u_n|)|\gd^\ga dx \\
&\leq g(s_0)\int_{G} \gd^\ga dx+\int_{E_s(u_n)} g\left(|\nabla u_n|\right)\gd^\ga dx\\
&\leq g(s_0)\int_{G}{}\gd^\ga dx-\int_{s_0}^{\infty}g(s)de_n(s).
\end{align*}
But
\begin{align*}
-\int_{s_0}^{\infty}g(s)de_n(s)
\leq g(s_0)e_n(s_0)&+c \|\gn\|_{\mathfrak M(\partial\xO)}^{q_\mu}
\int_{s_0}^{\infty}s^{-q_\mu}dg(s)\leq c \|\gn\|_{\mathfrak M(\partial\xO)}^{q_\mu}\int_{s_0}^{\infty}s^{-q_\mu-1}g(s)ds.
\end{align*}
Thus we have proved
\begin{align*}
\int_{G}|g_n(|\nabla u_n|)|\ei dx&\leq g(s_0)\int_{G}\gd^\ga dx+ c\|\gn\|_{\mathfrak M(\partial\xO)}^{q_\mu}\int_{s_0}^{\infty}s^{-q_\mu-1}g(s)ds.
\end{align*}

We obtain easily, using \eqref{G1} and fixing $s_0$ first, that for any $\ge>0$, there exists $\gk>0$ such that
 \begin{equation}\label{M14}
\int_{G} \gd^\ga dx\leq \gk \Longrightarrow \int_{G} |g_n(|\nabla u_n|)|\gd^\ga dx\leq\ge.
\end{equation}
Thus we invoke Vitali convergence theorem to derive that $g_n(|\nabla u_n|)\rightarrow g(|\nabla u|)$ in $L^1(\xO,\gd^\ga )$. From \eqref{bound11} we deduce that $0\leq u_n\rightarrow u$ in $L^1(\xO,\xd^\xa )$. Letting $n\to\infty$ in identity \eqref{M99}, we deduce that $u$ is a weak solution of \eqref{BP}.
\end{proof}

The next results asserts that any weak solution $u$ of \eqref{BP} behaves like $\BBK_\mu^\Gw[\nu]$ on the boundary.

\begin{proof}[{\bf Proof of Proposition C}] Since $u$ is a nonnegative solution of \eqref{BP}, formulation \eqref{uGK} holds. By a similar argument as in \cite[Proposition I]{MaNg}, $\BBG_\mu^\Gw[g(|\nabla u|)]$ is an $L_\mu$ potential (i.e. $\BBG_\mu^\Gw[g(|\nabla u|)]$ does not dominate any positive $L_\mu$ harmonic function). Consequently, in view of \cite[Theorem 2.6]{MaNg},
$$ \lim_{x \to y}\frac{\BBG_\mu^\Gw[g(|\nabla u|)](x)}{\BBK_\mu^\Gw[\nu](x)}= 0 \quad \text{non tangentially},\, \text{for } \nu\text{-a.e. } y \in \prt \Gw.
$$
This and \eqref{uGK} imply \eqref{nontang}.
\end{proof}

\subsection{Regularity} This subsection is devoted to the regularity property of distributional solutions.
\begin{definition} \label{subsupD} Let $D$ be a subdomain of $\Gw$. A function $u$ is called a (distributional) subsolution (resp. supersolution) of
	\bel{eqD} - L_\mu u + g(|\nabla u|) = 0 \quad \text{in } D\ee
if $u, g(|\nabla u|) \in L_{loc}^1(D)$ and
\bel{subD} \int_D (- u \,L_\mu \zeta + g(|\nabla u|)\zeta )dx \leq  (\text{resp.} \geq)\,\,\, 0 \quad \forall \zeta \in C_c^\infty(D),\, \zeta \geq 0. \ee
A (distributional) solution in $D$ is a distributional subsolution and supersolution in $D$.
\end{definition}

\begin{lemma} \label{regularity}
Let $1\leq q< \frac{N}{N-1}$ and assume that $g$ is a locally Lipschitz function satisfying
\be |g(t)|\leq C (t^q+1), \quad \forall t \geq 0,\label{G5} \ee
for some $C>0$.
If $u$ is a distributional solution of \eqref{E} then $u\in C^2(\xO)$.
\end{lemma}
\begin{proof} Let $x_0\in\xO$ and $r_0=\gd(x_0)$. Let $G_i(x,y)$ and $K_i(x,y)$ be the Green kernel and Poisson Kernel of $-\xD$ in $B(x_0,\frac{r_0}{2^i})$ respectively.
Since $u,g(|\nabla u|)\in L^1(B(x_0,\frac{r_0}{2}))$, it follows that for any $x\in B(x_0,\frac{r_0}{2^i})$,
\bel{repi} \BAL
u(x)=\xm\int_{B(x_0,\frac{r_0}{2^i})} G_i(x,y)\frac{u(y)}{\gd(y)^2}dy &-\int_{B(x_0,\frac{r_0}{2^i})}G_i(x,y)g(|\nabla u(y)|)dy \\
&+\int_{\partial B(x_0,\frac{r_0}{2^i})} K_i(x,y)u(y)dS_y.
\EAL \ee
From the above formula, for any $x\in B(x_0,\frac{r_0}{2^i})$,
\bel{bound} \BAL
|\nabla u(x)|\leq \int_{B(x_0,\frac{r_0}{2^i})} |\nabla_x G_i(x,y)|\frac{|u(y)|}{\gd(y)^2}dy &+ \int_{B(x_0,\frac{r_0}{2^i})}|\nabla_x G_i(x,y)|g(|\nabla u(y)|)dy \\
&+\left|\nabla_x\int_{\partial B(x_0,\frac{r_0}{2^i})} u(y) K_i(x,y)dS_y\right|.
\EAL \ee
Combining \eqref{repi}, \eqref{bound}, \cite[Theorem 2.5 and Theorem 2.6]{BVi} and the fact that \bel{KinC2} \int_{\partial B(x_0,\frac{r_0}{2^i})} K_i(x,y)u(y)dS_y\in C^2(B(x_0,\frac{r_0}{2^i})),
\ee
we deduce that $u,\, |\nabla u| \in L^p(B(x_0,\frac{r_0}{2^2}))$ for any $q<p<\frac{N}{N-1}$.

Let $q<p<\frac{N}{N-1}$ and put $p_i:=(\frac{p}{q})^{i-1}p,\;i\in\mathbb{N}$.

\textbf{Claim:} There hold
\bel{induction} u, |\nabla u| \in L^{p_i}(B(x_0,\frac{r_0}{2^{i+1}})) \quad \forall i \in \BBN. \ee

We will prove the claim  by induction. Indeed, \eqref{induction} holds for $i=1$. Suppose that \eqref{induction} true for some $i \in \BBN$. We will show that \eqref{induction} holds for $i+1$. By \cite[Proposition 2.1]{BVi}, for every $x,y\in  B(x_0,\frac{r_0}{2^i}),\;x\neq y$,
\bel{Gi}
G_i(x,y)\leq C_i |x-y|^{-N+2}\quad\text{and}\quad |\nabla_x G_i(x,y)| \leq C_i |x-y|^{-N+1}.
\ee
From \eqref{repi}--\eqref{Gi} and the assumption on $g$,  we can easily obtain the following estimates for any $x\in B(x_0,\frac{r_0}{2^{i+2}})$
\begin{align}\label{bound2}
|u(x)|&\leq C_i\left(\int_{B(x_0,\frac{r_0}{2^{i+1}})}\frac{|u(y)|}{|x-y|^{N-2}}dy +\int_{B(x_0,\frac{r_0}{2^{i+1}})}\frac{|\nabla u(y)|^q}{|x-y|^{N-2}}dy+1\right),\\
|\nabla u(x)|&\leq C_i\left(\int_{B(x_0,\frac{r_0}{2^{i+1}})}\frac{|u(y)|}{|x-y|^{N-1}}dy +\int_{B(x_0,\frac{r_0}{2^{i+1}})}\frac{|\nabla u(y)|^q}{|x-y|^{N-1}}dy+1\right).\label{bound3}
\end{align}
By \eqref{bound2} and Holder inequality, for every $x \in B(x_0,\frac{r_0}{2^{i+2}})$,
\begin{align*}
|u(x)|^{p_{i+1}} \leq C_i \left(\int_{B(x_0,\frac{r_0}{2^{i+1}})}\frac{|u(y)|^{p_i}}{|x-y|^{p(N-2)}}dy +\int_{B(x_0,\frac{r_0}{2^{i+1}})}\frac{|\nabla u(y)|^{p_i}}{|x-y|^{p(N-2)}}dy+1\right).
\end{align*}
By integrating over $B(x_0,\frac{r_0}{2^{i+2}})$ and keeping in mind that $p<\frac{N}{N-1}$, we obtain
\begin{align*}
\int_{B(x_0,\frac{r_0}{2^{i+2}})}|u(x)|^{p_{i+1}}dx&\leq C_i \Big( \int_{B(x_0,\frac{r_0}{2^{i+2}})}\int_{B(x_0,\frac{r_0}{2^{i+1}})}\frac{|u(y)|^{p_i}}{|x-y|^{p(N-2)}}dydx \\ &+ \int_{B(x_0,\frac{r_0}{2^{i+2}})}\int_{B(x_0,\frac{r_0}{2^{i+1}})}\frac{|\nabla u(y)|^{p_i}}{|x-y|^{p(N-2)}}dydx+1 \Big) \\
&\leq C_i.
\end{align*}
Similarly, since $p<\frac{N}{N-1}$, we deduce from \eqref{bound3} that
\begin{align*}
\int_{B(x_0,\frac{r_0}{2^{i+2}})}|\nabla u(x)|^{p_{i+1}}dx&\leq C_i\Big( \int_{B(x_0,\frac{r_0}{2^{i+2}})}\int_{B(x_0,\frac{r_0}{2^{i+1}})}\frac{|u(y)|}{|x-y|^{p(N-1)}}dydx \\ &+\int_{B(x_0,\frac{r_0}{2^{i+2}})}\int_{B(x_0,\frac{r_0}{2^{i+1}})}\frac{|\nabla u(y)|^q}{|x-y|^{p(N-1)}}dydx+1\Big)\\
&\leq C_i.
\end{align*}
Therefore \eqref{induction} holds true for $i+1$. Thus we have proved the claim.

Now fix $i$ large enough such that $\frac{qp}{p-1}<p_{i-1}.$ Then $u,|\nabla u|\in L^{p_{i-1}}(B(x_0,\frac{r_0}{2^{i}}))$. We will estimate the terms on the right hand-side of \eqref{repi}. By \eqref{Gi} and Holder inequality,
\bel{mi1} \BAL \int_{B(x_0,\frac{r_0}{2^i})} G_i(x,y)\frac{|u(y)|}{\gd(y)^2}dy
&\leq C_i \left(\int_{B(x_0,\frac{r_0}{2^{i}})}\frac{1}{|x-y|^{p(N-2)}}dy\right)^{\frac{1}{p}} \left(\int_{B(x_0,\frac{r_0}{2^{i}})}|u|^{\frac{p}{p-1}}dy\right)^{\frac{p-1}{p}} \leq C_i.
\EAL \ee
Similarly, by  \eqref{Gi}, Holder inequality and the assumption \eqref{G5},
\bel{mi2} \BAL
\int_{B(x_0,\frac{r_0}{2^{i}})} &G_{i}(x,y)g(|\nabla u(y)|)dy\leq C_i \left( \int_{B(x_0,\frac{r_0}{2^{i}})}\frac{|\nabla u(y)|^q}{|x-y|^{N-1}}dy +1 \right)\\
&\leq C_i\left(\int_{B(x_0,\frac{r_0}{2^{i}})}\frac{1}{|x-y|^{p(N-1)}}dy\right)^{\frac{1}{p}}\left(\int_{B(x_0,\frac{r_0}{2^{i}})}|\nabla u|^{\frac{qp}{p-1}}dy\right)^{\frac{p-1}{p}} \leq C_i.
\EAL \ee
Combining \eqref{mi1}, \eqref{mi2} and \eqref{KinC2}, we obtain that $u\in L^{\infty}(B(x_0,\frac{r_0}{2^{i+1}}))$. By a similar argument, one can show that $|\nabla u| \in L^{\infty}(B(x_0,\frac{r_0}{2^{i+1}}))$. Thus the desired regularity result follows by standard elliptic regularity theory.
\end{proof}

\subsection{Comparison principle}

\begin{lemma} \label{positive} Let $g$ be a locally Lipschitz function. If $u\in C^2(\xO)$ is a nonnegative solution
\be\label{E}
-L_\mu u+g(|\nabla u|)=0 \quad\text{in } \Gw.
\ee
and there exists $x_0\in\xO$ such that $u(x_0)=0$ then $u\equiv 0$.
\end{lemma}
\begin{proof} Since $g$ is locally Lipschitz function and $u\in C^2(\xO)$ we can write
$$g(|\nabla u|)=\textbf{b}(x)\nabla u \quad\text{a.e. in}\; \xO$$
where ${\bf b}$ has the following property: for any $\gb \in (0,\gb_0)$, there exists $C_\gb$ such that
$$\sup_{x\in D_\xb}|\textbf{b}(x)|\leq C_\gb.$$
Let $\xb \in (0,\gb_0)$ be small enough such that $x_0\in D_\xb$. Note that $u$ is a nonnegative solution of
$$
-\Gd u+\textbf{b}\nabla u= \frac{\mu}{\gd^2}u\geq0 \quad\text{in } D_\xb.
$$
Thus, by the maximum principle, $u$ cannot achieve a nonpositive minimum in $D_\xb.$ Thus the result follows straight forward.
\end{proof}


Next we state the comparison principle for \eqref{eqD}.

\begin{lemma} \label{comp1} Let $g$ be a locally Lipschitz function and satisfy \eqref{G3}. We assume that $D \sbs \Gw$ and $u_1, u_2\in C^2(D)$ are respectively nonnegative subsolution and positive supersolution of \eqref{eqD} in $D$ such that
\bel{bdryD} \lim_{x \to \prt D}(u_1(x) - u_2(x))<0. \ee
Then $u_1 \leq u_2$ in $D$.
\end{lemma}
\begin{proof}
Suppose by contradiction that
$$ \sup_{x \in D}\frac{u_1(x)}{u_2(x)}=:m > 1. $$

By \eqref{bdryD}, we deduce that there exists $x_0 \in D$ such that
$$ \frac{u_1(x_0)}{u_2(x_0)}=\sup_{x \in D}\frac{u_1(x)}{u_2(x)}=m. $$

Let $r>0$ be such that $B(x_0,r)\subset D$. Then we see that
\begin{align*}
- \Gd (m^{-1}u_1-u_2)+ g(m^{-1}|\nabla u_1|) - g(|\nabla u_2|)\leq \frac{\mu}{\gd^2} (m^{-1}u_1-u_2) \leq0\quad \text{in } B(x_0,\frac{r}{2}).
\end{align*}

Since $g$ is locally Lipschitz, we can write
$$g(m^{-1}|\nabla u_1|) - g(|\nabla u_2|)=\textbf{b}\nabla \left(m^{-1}u_1-u_2\right) \quad\text{in  } B(x_0,\frac{r}{2}),$$
where ${\bf b}$ satisfies $\sup_{x\in B(x_0,\frac{r}{2})}|\textbf{b}(x)|\leq C$.
Hence,
\begin{align*}
- \Gd (m^{-1}u_1-u_2)+ {\bf b}\nabla \left(m^{-1}u_1-u_2\right)\leq 0 \quad \text{in } B(x_0,\frac{r}{2}),
\end{align*}
and by maximum principle $m^{-1}u_1-u_2$ can not achieve a non-negative maximum in $B(x_0,\frac{r}{2}).$
This is a contradiction. Thus $u_1 \leq u_2$ in $D$.
\end{proof}

Next we will prove the comparison principle for \eqref{BP}.

In order to demonstrate Theorem D, we need the following auxiliary result.

\begin{lemma} \label{lem2} Let $\gt \in \GTM(\Gw,\gd^\ga)$ and $v\geq 0$ satisfies
\bel{b8} \left\{ \BAL - L_\mu v &\leq \gt \quad \text{in } \Gw, \\
 \tr(v) &=0
\EAL \right. \ee
Then for any $1<q<q_\mu$, there exists a constant  $c=c(N,\Gw,\mu)$ such that
$$\norm{\nabla v}_{L^q(\Gw,\gd^\ga)} \leq c\norm{\gt}_{\GTM(\Gw,\gd^\ga)}. $$
\end{lemma}
\begin{proof}
We notice that $(v - \BBG_\mu^\Gw[\gt])_+$ is a nonnegative $L_\mu$ subharmonic function with $\tr((v - \BBG_\mu^\Gw[\gt])_+)=0$. By Propositions \ref{PropA}, $(v - \BBG_\mu^\Gw[\gt])_+ = 0$, i.e. $v \leq \BBG_\mu^\Gw[\gt]$ a.e. in $\Gw$. Put $\tl v=\BBG_\mu^\Gw[\gt] - v$ then $\tl v$ is a nonnegative $L_\mu$ superharmonic function in $\Gw$ and $\tr(\tl v)=0$. Due to Proposition \ref{PropA}, there exists $\tl \gt \in \GTM^+(\Gw,\gd^\ga)$ such that
\bel{eqtlv} -L_\mu \tl v = \tl \gt. \ee
This implies that $\tl v = \BBG_\mu^\Gw[\tl \gt]$ and hence $v=\BBG_\mu^\Gw[\gt -\tl \gt]$. Therefore, by Lemma \ref{gradGK2},
\bel{luka} \| \nabla v \|_{L^q(\Gw,\gd^\ga)} \leq c\| \gt -\tl \gt \|_{\GTM(\Gw,\gd^\ga)} \leq c(\| \gt  \|_{\GTM(\Gw,\gd^\ga)} + \| \tl \gt \|_{\GTM(\Gw,\gd^\ga)}).
\ee
By using $\BBG_\mu^\Gw[1]$ as a test function for \eqref{eqtlv}, keeping in mind the estimate
\bel{G11} c^{-1}\gd^\ga \leq \BBG_\mu^\Gw[1] \leq c \gd^\ga \quad \text{in } \Gw, \ee
we obtain
\bel{lak} \| \tl \gt \|_{\GTM(\Gw,\gd^\ga)} \leq c\int_\Gw \BBG_\mu^\Gw[1] d\tl \gt = c\int_\Gw \tl v dx  \leq c\int_{\Gw}\BBG_\mu^\Gw[\gt]dx \leq c' \| \gt \|_{\GTM(\Gw,\gd^\ga)}.
\ee
Combining \eqref{luka} and \eqref{lak}, we deduce \eqref{b8}.
\end{proof}

We turn to the

\begin{proof}[{\bf Proof of Theorem D}] Since $u_i$ is a solution of \eqref{BP}, $g(|\nabla u_i|) \in L^1(\Gw,\gd^\ga)$, $i=1,2$. Moreover, from Lemma \ref{gradGK2} and Lemma \ref{gradP}, we deduce that
$$ \norm{\nabla u_i}_{L^q(\Gw,\gd^\ga)} \leq c_1(\norm{g(|\nabla u_i|)}_{L^1(\Gw,\gd^\ga)}+ \norm{\nu_i}_{\GTM(\prt \Gw)}).
$$

Without loss of generality we assume that $\xn_2\neq0,$ thus by Lemma \ref{positive}  $u_2(x)>0$ for any $x\in \xO.$ In addition, by Lemma \ref{regularity}, $u_i\in C^2(\xO).$ Finally by the representation formula we have
$$u_i+\BBG_{\xm}^\xO[g(|\nabla u_i|)]=\BBK_{\xm}^\xO[\nu_i],\quad i=1,2.$$

Let $0<\xe\leq1,$ then
\begin{align*}
(\xe u_1-u_2)_+\leq \left(\BBG_{\xm}^\xO[g(|\nabla u_2|)]-\xe\BBG_{\xm}^\xO[g(|\nabla u_1|)]\right)_+\leq \BBG_{\xm}^\xO[\left|g(|\nabla u_2|)-\xe g(|\nabla u_1|)\right|]=:v,
\end{align*}
which implies
$$\tr((\xe u_1-u_2)_+)\leq \tr(v)=0.$$
Due to \eqref{G3},  $\xe u_1$ is a subsolution of \eqref{N}. Also since $u_i\in C^2(\xO)$ and $u_2>0$ in $\xO$, it follows that
$$\sup_{x\in D_\xb}\frac{u_1}{u_2}= C_\xb<\infty,$$
where $\xb>0$ is small enough. Without loss of generality we assume that $C_\xb>1.$ Set $\xe_\xb=\frac{1}{C_\xb}<1$. Then $\vge_\gb u_1 - u_2 \leq 0$ in $D_\gb$.

\noindent \textbf{Claim:} For any $\gb>0$ small enough, there holds
\be
\xe_\xb u_1-u_2<0 \quad \text{in }  D_\xb.\label{megisto}
\ee
Indeed, from \eqref{G3}, we observe that
\begin{align*}
- \Gd (\xe_\xb u_1-u_2)+ g(\xe_\xb|\nabla u_1|) - g(|\nabla u_2|)\leq \frac{\mu}{\gd^2} (\xe_\xb u_1-u_2) \leq0 \quad \text{in } D_\xb.
\end{align*}
Since $g$ is locally Lipschitz, there holds
$$g(\xe_\xb|\nabla u_1|) - g(|\nabla u_2|)={\bf b} \nabla \left(\xe_\xb u_1-u_2\right) \quad\text{a.e. in}\; D_\xb,$$
with the estimate $\sup_{x\in D_\xb}|\textbf{b}(x)|\leq C$.
Hence
\begin{align*}
- \Gd (\xe_\xb u_1-u_2)+ {\bf b}\nabla \left(\xe_\xb u_1-u_2\right)&\leq 0 \quad \text{in }  D_\xb.
\end{align*}
By the maximum principle $\xe_\xb u_1-u_2$ can not achieve a nonnegative maximum in $D_\xb$ and thus the claim follows.

Due to Kato's inequality \cite{MVbook}, we get
\bel{b7} -L_\mu(\xe_\xb u_1-u_{2})_+ \leq (g (|\nabla u_2|) -g (\xe_\xb |\nabla u_1|))\chi_{_{E_{\xb}}}, \ee
where $E_{\xb}=\{x \in \Gw: \xe_\xb u_{1}-u_{2}>0\}$. By \eqref{megisto} we derive that
$E_\xb\subset\xO_\xb$.

Applying Lemma \ref{lem2} and Holder's inequality, thanks to \eqref{G2}, we get
\bel{b9} \BAL \int_{\Gw} |\nabla(\xe_\xb u_{1}-&u_{2})_+ |^q\gd^{\ga} dx \leq c\left(\int_{\Gw} |g (|\nabla u_2|) -g (\xe_\xb |\nabla u_1|)|\chi_{_{E_{\xb}}}\gd^\ga dx\right)^q\\
&\leq c\left(\int_{E_{\xb}} (\xe^{q-1}_\xb|\nabla u_1|^{q-1}+|\nabla u_2|^{q-1})|\nabla(\xe_\xb u_1-u_2)|\gd^{\ga} dx\right)^q\\
&\leq c\left(\int_{E_{\xb}} (\xe^q_\xb|\nabla u_1|^q+ |\nabla u_2|^q)\gd^{\ga} dx\right)^{q-1}\int_{E_{\xb}} |\nabla(\xe_\xb u_1-u_2)|^q\gd^{\ga} dx.
\EAL \ee
Since $E_\xb\subset \xO_\gb$ and $|\nabla u_i| \in L^q(\Gw,\gd^\ga)$, we can choose $\xb_*$ small enough such that
\bel{bbb}
c\left(\int_{E_{\xb_*}} (|\nabla u_1|^q+ |\nabla u_2|^q)\gd^{\ga} dx\right)^{q-1}<\frac{1}{4}.
\ee
By the above inequality and  \eqref{b9} we obtain that
$$\nabla(\xe_{\gb_*} u_1-u_2)_+=0\Rightarrow (\xe_{\gb_*} u_1-u_2)_+=c_*,$$
for some constant $c_*\geq0$ and since $(\xe_{\gb_*} u_1-u_2)_+=0$ on $\overline{D}_{\xb_*}$ we have that $c_*=0$, namely $\vge_{\gb_*}u_1 \leq u_2$ in $\Gw$. As a consequence,
\bel{vuq}\sup_{x\in\xO}\frac{u_1(x)}{u_2(x)}=\sup_{x\in D_{\xb_*}}\frac{u_1(x)}{u_2(x)}=\sup_{x\in \partial D_{\xb_*}}\frac{u_1(x)}{u_2(x)}=\xe_{\xb_*}^{-1}>1.
\ee
This implies the existence of $x_*\in \partial D_{\xb_*}$ such that
\bel{iml} (\xe_{\gb_*} u_1-u_2)(x_*)=0. \ee
Next we take $\gb<\gb_*$, then  $\vge_\gb \leq \vge_{\gb_*}$. On the other hand, we infer from \eqref{vuq} that $\vge_\gb \geq \vge_{\gb_*}$ and hence $\vge_\gb = \vge_{\gb_*}$. Therefore \eqref{iml} contradicts \eqref{megisto}.
\end{proof}
\subsection{Some estimates}

\begin{lemma} \label{apriori} Assume $g(t)=t^q$ with $1<q<\frac{N}{N-1}$. If $u$ is a nonnegative solution of
\bel{Epower} -L_\mu u + |\nabla u|^q = 0 \quad \text{in } \Gw \ee
then
\begin{align}
\label{apriori1} u(x) \leq C \gd(x)^{-\frac{2-q}{q-1}} + M_{\gb_0} \quad \forall x \in \Gw, \\
\label{apriori2} |\nabla u(x)| \leq C'\gd(x)^{-\frac{1}{q-1}} \quad \forall x \in \Gw,
\end{align}
where $M_{\gb_0}:=\sup_{\overline D_{\gb_0}}u$, $C=C(N,\mu,q,\gb_0,M_{\xb_0})$ and $C'=C'(N,\mu,q,\gb_0,M_{\xb_0})$.
\end{lemma}
\Proof For $\gb \in (0,\gb_0)$, put
$$ w_\gb(x) = C (\gd(x)-\gb)^{-\frac{2-q}{q-1}} + M_{\gb_0},  \quad x \in D_\gb. $$
By a simple computation, we deduce that for $C_1=C_1(N,\mu,\gb_0,q)$ large enough
$$ -L_\mu w_\gb + |\nabla w_\gb|^q \geq 0 \quad \text{in } D_\gb \sms D_{\gb_0}. $$
By Lemma \ref{comp1}, $u \leq w_\gb$ in $D_\gb \sms D_{\gb_0}$. Consequently, by letting $\gb \to 0$, we obtain  \eqref{apriori1}.

Next we prove \eqref{apriori2}. Fix $x_0 \in \Gw$ and set $d_0=\frac{1}{3}\gd(x_0)$, $y_0=\frac{1}{d_0}x_0$ and
$$ \quad m_{0}=\max\{u(x): x \in B_{2d_0}(x_0)\}, \q
v(y)=\frac{u(x)}{m_0},\quad  y=\frac{1}{d_0}x \in B_2(y_0). $$
Then $\max\{v(y): y \in B_2(y_0)\}=1$, $m_0^{q-1}d_0^{2-q}<C(N,\mu,q,\gb_0,M_{\xb_0})<\infty$, and
$$ -\Gd v - \frac{\mu}{\dist(y,\prt(d_0^{-1}\Gw))^2}v +   m_0^{q-1}d_0^{2-q}|\nabla v|^q = 0 \quad \text{in } B_2(y_0).$$
By \cite{La},  there exists a positive constant $c=c(N,\mu,q,\gb_0,M_{\xb_0})$ such that $$\max_{B_1(y_0)}|\nabla v| \leq c.$$
Consequently,
$$ \max_{B_{d_0}(x_0)}|\nabla u| \leq cd_0^{-1}\max_{B_{2d_0}(x_0)}u $$
which implies \eqref{apriori2}. \qed

\section{Isolated boundary singularities} In this section, we assume that $0 \in \prt \Gw$ and study the behavior near $0$ of solutions of \eqref{Epower}  which vanish on $\prt \Gw \setminus \{0\}$.
\subsection{A priori estimates} We first establish pointwise a priori estimates for solutions with isolated singularity at $0$, as well as their gradient.


\begin{proof}[{\bf Proof of Proposition E}] We first prove \eqref{3.4.24}. \medskip
	
\noindent \textbf{Step 1.} Let $\xb_0$ be the constant in Proposition \ref{barr}. Let $x_i\in \partial\xO$ be such that $|x_i|\geq \frac{\xb_0}{16},$ $$\partial\xO\subset B(0,\frac{\xb_0}{4})\cup\bigcup_{i=1}^nB(x_i,\frac{\xb_0}{32})=:A,$$
for some $n\in\mathbb{N}$. Notice that there exists a constant $\xe_0=\xe_0(\xb_0)>0$ such that $\dist(\partial A,\partial\xO)>\xe_0.$

Let $w_i$ be the function constructed in Proposition \ref{barr} in $B(x_i,\frac{\xb_0}{16})$ for $R=\frac{\xb_0}{16},\;i=1,...,n.$
Then by the maximum principle (see \cite[Proposition 2.13 and 2.14]{GkV}), we have that
$$u(x)\leq w_i(x),\quad\forall x\in B(x_i,\frac{\xb_0}{16}),\;i=1,...,n. $$
As a consequence, there is a positive constant $C_0=C_0(N,\mu,q,\xO,\xb_0)$ such that $$u(x)\leq C_0 \quad\forall x\in \cup_{i=1}^nB(x_i,\frac{\xb_0}{32}).$$

Set $$v(x):=C_1\left(|x|-\frac{\xb_0}{4}\right)^{-\frac{2-q}{q-1}} \quad \text{with }
C_1=3C_0(\sup_{x\in\xO}|x|)^{\frac{2-q}{q-1}}.$$

We will show that $v(x)\geq u(x)$ for every  $x\in \xO\setminus A$. Indeed, by a direct computation,
\bel{lem1} \BAL
-\xD v &=C_1\frac{2-q}{q-1}\left((N-1)|x|\left(|x|-\frac{\xb_0}{4}\right)^{-1}+(q-1)^{-1}\left(|x|-\frac{\xb_0}{4}\right)^{-2}\right)v\\
&\geq C_1\frac{2-q}{(q-1)^2}\left(|x|-\frac{\xb_0}{4}\right)^{-\frac{q}{q-1}}\quad\forall x\in\xO\setminus\overline{A},
\EAL \ee
\bel{lem22}
|\nabla v|^q=C_1^q\left(\frac{2-q}{q-1}\right)^q\left(|x|-\frac{\xb_0}{4}\right)^{-\frac{q}{q-1}} \quad\forall x\in\xO\setminus\overline{A},
\ee
and
\bel{lem3} \xm\frac{v(x)}{\gd(x)^2}\leq C_1\xe_0^{-2}(\sup_{x\in\xO}|x|)^{-2}\left(|x|-\frac{\xb_0}{4}\right)^{-\frac{q}{q-1}} \quad\forall x\in\xO\setminus\overline{A}.
\ee
Gathering estimates \eqref{lem1}--\eqref{lem3} leads to, for $C_1=C_1(N,\mu,q,\xb_0,\Gw)>0$ large enough,
$$\BAL
-L_\xm v+|\nabla v|^q &\geq \left(|x|-\frac{\xb_0}{4}\right)^{-\frac{q}{q-1}} \left[-C_1\frac{2-q}{(q-1)^2}-C_1 \vge_0^{-2}(\sup_{x\in\xO}|x|)^{-2}+C_1^q\left(\frac{q-2}{q-1}\right)^q\right]\\
&\geq0 \quad\forall x\in\xO\setminus\overline{A}.
\EAL $$
Moreover, we have
$$ \limsup_{x \to \prt (\Gw \setminus \overline A)}(u-v) < 0. $$
By Lemma \ref{comp1} we deduce that $u \leq v$ in $\xO\setminus\overline{A},$ which implies that
$$u(x)\leq C(N,\mu,q,\Gw,\gb_0) \quad\forall x\in D_{\gb_0}.$$
Thus by Lemma \ref{apriori} there exists $C_2=C_2(\Gw,N,\mu,q,\gb_0)>0$ such that
\be
u(x)\leq C_2\gd(x)^{-\frac{2-q}{q-1}} \quad\forall x\in \xO.\label{3.4.1}
\ee \smallskip

\noindent \textbf{Step 2.}  For $\ell>0$, put
\bel{scal} T_\ell[u](x):=\ell^{\frac{2-q}{q-1}}u(\ell x), \quad x \in \Gw^\ell:=\ell^{-1}\Gw. \ee
Let $\xi\in\partial\xO\setminus \{0\}$
and put $d=d(\xi):= \frac{1}{2} |\xi|$. We assume that $d \leq 1$. Denote $u_d: =T_{d}[u]$ then $u_d$ is a solution of \eqref{Epower} in $\Gw^{d}=\frac{1}{d}\Gw$.
Let $R_0=\frac{\xb_0}{16},$ where $\xb_0$ is the constant in Proposition \ref{barr}.

Then the solution $w_{\xi,\frac{3R_0}{4}}$ mentioned
in Proposition \ref{barr} satisfies
$$
u_d(y) \leq w_{\xi,\frac{3R_0}{4}}(y)\quad \forall y\in B_{\frac{3R_0}{4}}(\xi)\cap\xO^{d} .$$
Thus $u_d$ is bounded above in $B_{\frac{3R_0}{5}}(\xi)\cap\xO^{d}$ by a constant $C>0$ depending only on $N, q, \xm $
and the $C^2$ characteristic of $\xO^{d}$. As $d \leq 1$ a $C^2$ characteristic of $\xO$
is also a $C^2$ characteristic of $\xO^{d}$, therefore the constant $C$ can be taken to be independent of $\xi$. We note here that the constant $R_0 \in (0,1)$ depends on $C^2$ characteristic of $\xO$.

Now put
$$W_d(y):=\left\{ \BAL &\dist(y,\prt \Gw^{d})^{1-\ga} \qquad&& \text{if } \xk <\frac{1}{4}, \\
&\dist(y,\prt\Gw^{d})^{\frac{1}{2}} |\ln \dist(y,\prt\Gw^{d})| &&\text{if } \xk =\frac{1}{4}.
\EAL \right.
$$
Then we infer from \eqref{lama} that
$$\lim_{y\in\xO^{d},\;y\rightarrow P}\frac{u_d(y)}{W_d(y)}=0\qquad\forall P\in  B_{\frac{3R_0}{5}}(\xi) \cap \partial\xO^{d}.$$
Since $u_d$ is a positive $L_\xm$ subharmonic in $\Gw^{d}$, in view of the proof of \cite[Propositions 2.11 and 2.12]{GkV}, we deduce that
$$\frac{u}{\dist(\cdot,\prt \Gw^{d})^\xa}\in H_0^1(\xO^d\cap B_{\frac{11R_0}{20}}(\xi),\dist(\cdot,\prt \Gw^{d})^{2\xa}).$$
By proceeding as in the proof of \cite[Theorem 2.12 ]{FMT}, we deduce that there exists  $C=C(N,\mu,q)>0$ such that
\bel{sim}
u_d(y) \leq (\mathrm{dist}(y,\partial\xO^{d})^\xa\quad \forall y\in B_{\frac{R_0}{2}}(\xi)\cap\xO^{d}.\ee
Hence
\bel{sun1}
u(x)\leq \gd(x)^\xa d^{-\frac{2-q}{q-1}-\xa}\quad\forall x\in B_{d\frac{R_0}{2}}(\xi)\cap\xO.
\ee

Let $x\in \xO_{\frac{R_0}{2}}$ and $\xi$ be the unique point in $\partial\xO\setminus \{0\}$ such that $|x-\xi|=\gd(x)$. Put $d=\frac{1}{2}|\xi|$.

\textit{Case 1:} $|x|<\frac{1}{1+R_0}$. If $\gd(x)\leq \frac{R_0|x|}{16}$ then
\bel{sun2} \BAL &2d \leq \gd(x)+|x|\leq (1+  R_0/16)|x|<1, \\
&2d \geq |x| - \gd(x) \geq  (1- R_0/16)|x|.
\EAL \ee
As a consequence,
\bel{sun3} |x-\xi| =\gd(x) \leq \frac{R_0}{16}|x| \leq \frac{R_0}{8}\left( 1 - \frac{R_0}{4} \right)^{-1} d < \frac{R_0}{2}d.
\ee
In the last inequality we have used the fact that $R_0<1$. By combining \eqref{sun1}--\eqref{sun3}, we obtain
\bel{sun4} u(x) \leq C\gd(x)^{\xa}\left((1-R_0/4)|x|\right)^{-\frac{2-q}{q-1}-\xa}.
\ee
If   $\gd(x)> \frac{R_0|x|}{16}$ then by \eqref{3.4.1} we have
\bel{sun5} u(x) \leq C\gd(x)^{-\frac{2-q}{q-1}} \leq C\gd(x)^{\xa}\left(\frac{R_0}{4}|x|\right)^{-\frac{2-q}{q-1}-\xa}. \ee
From \eqref{sun4} and \eqref{sun5} we deduce that \eqref{3.4.24} holds for every $x\in\xO_{\frac{R_0}{2}}$ such that $|x| < \frac{1}{1 + R_0}$.

\textit{Case 2:} $|x|\geq \frac{1}{1 + R_0}$.  By an argument similar to the one used to obtain  \eqref{sim}, we can prove that
$$
u(x)\leq C \gd(x)^\ga \leq  C' \gd(x)^\ga |x|^{-\frac{2-q}{q-1}-\xa}\qquad\forall x\in B_{\frac{R_0}{2}}(\xi)\cap\xO.
$$

If $x\in \xO\setminus\xO_{\frac{R_0}{2}}$ then  \eqref{3.4.24} follows directly from \eqref{3.4.1}.

Next we prove \eqref{3.4.24*}. Let $x_0\in\xO$ such that $\gd(x_0)<\min(\frac{\xb_0}{36},1)$. By proceeding as in the proof of \eqref{apriori2}, we can deduce
$$ \max_{y\in B_{\frac{1}{3}\gd(x_0)}(x_0)}|\nabla u| \leq c\gd(x_0)^{-1}\max_{y\in B_{\frac{2}{3}\gd(x_0)}(x_0)}u. $$
This and \eqref{3.4.24} imply  \eqref{3.4.24*}.

If $\gd(x_0)\geq\min(\frac{\xb_0}{36},1)=:r_0,$ by \cite{La} and \eqref{3.4.24},  there exists a positive constant $c=c(N,\mu,q,\gb_0)$ such that $$\max_{y\in B_{\frac{r_0}{2}}(x_0)}|\nabla u| \leq c$$
and the result follows.
\end{proof}
\subsection{Weak singularities}

\begin{proof}[{\bf Proof of Theorem F}]  Let $u=u_{0,k}^\Gw$ be the solution of \eqref{Pc}. By Theorem B and Lemma \ref{positive}, $0<u \leq kK_\mu^\Gw(\cdot,0)$ in $\Gw$. Moreover,
\bel{formul} u + \BBG_\mu^\Gw[|\nabla u|^q] = k K_\mu^\Gw(\cdot,0).
\ee
By proceeding as in the proof of \eqref{apriori2}, we obtain
\bel{gary} |\nabla u(x)| \leq ck\gd(x)^{\ga-1} |x|^{2-N-2\ga} \quad \forall x\in \Gw. \ee
This follows that
\be\label{ole1}
 \BBG_\mu^\Gw[|\nabla u|^q](x) \leq ck^q  \int_{\Gw} \gd(y)^{(\ga-1)q}G_\mu^\Gw(x,y)|y|^{(2-N-2\ga)q}dy.
\ee

\noindent \textbf{Case 1: $\ga + (\ga-1)q \geq 0.$}

 By the assumption and \eqref{Gmu}, we have
\bel{gary2} \BBG_\mu^\Gw[|\nabla u|^q](x) \leq ck^q \gd(x)^\ga \int_{\Gw} |x-y|^{2-N-2\ga}|y|^{\ga-(N+\ga-1)q}dy.
\ee
Since $q<q_\mu$, it follows that
\bel{AB7} \int_{\Gw} |x-y|^{2-N-2\ga}|y|^{\ga - (N+1+\ga)q}dy \leq c |x|^{2-\ga-(N-1+\ga)q}.
\ee
Combining \eqref{gary2}, \eqref{AB7} and \eqref{Kmu} yields
\bel{gary4} \BBG_\mu^\Gw[|\nabla u|^q](x) \leq ck^q  |x|^{N+\ga-(N+\ga-1)q}K_\mu^\Gw(x,0). \ee
As a consequence,
\bel{AB4} \lim_{|x| \to 0}\frac{\BBG_\mu^{\Gw}[|\nabla u|^q](x)}{K_\mu^\Gw(x,0)}=0. \ee

\noindent \textbf{Case 2: $-1+\xa<\ga + (\ga-1)q <0.$}

Note that $q_\xm<\frac{1}{1-\xa}.$ By \eqref{ole1} and \eqref{Gmu} we have
\be\label{ole2}
 \BBG_\mu^\Gw[|\nabla u|^q](x) \leq ck^q  \int_{\Gw} \gd(y)^{(\ga-1)q}F_\xm(x,y)|y|^{(2-N-2\ga)q}dy,
\ee
where
$$F_\xm(x,y):=\abs{x-y}^{2-N}\min\left\{1,\gd(x)^{\ga  }\gd(y)^{\ga  }\abs{x-y}^{-2\ga  }\right\} \; \forall x,y \in \Gw,\, x \neq y.$$

Let $\xb \in (0,\gb_0)$ such that $\xd\in C^2(\overline{\xO_\xb}).$  We consider the cut-off function $\xf\in C^\infty(\overline{\xO_\frac{\xb}{2}}),$ such that $0\leq\xf\leq1,$ $\xf=1$ in  $\xO_\frac{\xb}{4}$ and $\xf=0$ in $\xO\setminus\overline{\xO_\frac{\xb}{2}}.$
Then
\bel{le1} \BAL
\int_\xO \xd(y)^{(\ga-1)q   }F_\xm(x,y)|y|^{-q(N+2\ga   -2)}&dy=\int_\xO \xd(y)^{(\ga-1)q  }F_\xm(x,y)|y|^{-q(N+2\ga   -2)}\xf(y)dy\\
&+\int_\xO \xd(y)^{(\ga-1)q   }F_\xm(x,y)|y|^{-q(N+2\ga   -2)}(1-\xf(y))dy.
\EAL \ee
By the definition of $\xf$ and using the inequality
\be
F_\xm(x,y)\leq\gd(x)^{\ga  }\abs{x-y}^{2-N-\ga},\label{32}
\ee
we obtain
\bel{le2} \BAL
\int_\xO \xd(y)^{(\ga-1)q   }&F_\xm(x,y)|y|^{-q(N+2\ga   -2)}(1-\xf(y))dy\\
&=\int_{\xO\setminus\xO_\frac{\xb}{4}} \xd^{(\ga-1)q   }(y)F_\xm(x,y)|y|^{-q(N+2\ga   -2)}(1-\xf(y))dy\\
&\leq C(\xb,N,\mu,  q)\xd(x)^{\ga   }.
\EAL \ee
Let $\widetilde{\xb} \in (0,\frac{\xb}{4})$ be such that $|x-y|>r_0>0$ for any $y\in \xO_{\widetilde{\xb}}$. Let $\xe>0$ be such that
$$q(N+\ga   -1)=N+\ga -\xe$$
and
$0<\tilde{\xe}<\xe$ be such that $(\ga-1)q   +1-\tilde{\xe}>0$.
Then by \eqref{32}, we have
\begin{align*}
\int_{\Gs_{\widetilde{\xb}}}\xd(y)^{(\ga-1)q   +1}&F_\xm(x,y)|y|^{-q(N+2\ga   -2)}dS(y)\\
&\leq\xd(x)^\ga  r_0^{2-N-2\ga   } \int_{\Gs_{\widetilde{\xb}}} \xd(y)^{\tilde{\xe}}|y|^{-q(N+\ga   -1)+\ga+1-\tilde{\xe}}dS(y)\\
&=\xd(x)^\ga  r_0^{2-N-2\ga   } \int_{\Gs_{\widetilde{\xb}}} \xd(y)^{\tilde{\xe}}|y|^{-N+1+ (\xe-\tilde{\xe})}dS(y).
\end{align*}
Note that by the choice of $\tilde{\xe}$, $N-2-N+1+ (\xe-\tilde{\xe})>-1$, which implies
$$\sup_{\widetilde{\xb}\in (0,\frac{\xb}{4})}\int_{\Gs_{\widetilde{\xb}}}|y|^{-N+1+(\xe-\tilde{\xe})}dS(y)<C.$$
Combining the above estimates, we deduce
\be
\lim_{\widetilde{\xb}\rightarrow0}\int_{\Gs_{\widetilde{\xb}}}\xd(y)^{(\ga-1)q   +1}F_\xm(x,y)|y|^{-q(N+2\ga   -2)}dS(y)=0. \label{in13}
\ee

Now note that
$$ \BAL
&-\int_{\xO_{\xb}} \nabla  \xd(y)\nabla F_\xm(x,y)\xd(y)^{(\ga-1)q   +1}|y|^{-q(N+2\ga   -2)}\xf(y)dy\\ \nonumber
&=(N-2) \int_{\xO_{\xb}} \frac{\nabla  \xd(y)\cdot(x-y)}{|x-y|^{N}}\min\left\{1,\frac{\gd(x)^{\ga  }\gd(y)^{\ga  }}{\abs{x-y}^{2\ga  }}\right\}\xd(y)^{(\ga-1)q   +1}|y|^{-q(N+2\ga   -2)}\xf(y)dy\\
&-\int_{\xO_{\xb}}
\nabla\xd(y)\nabla\left(\min\left\{1,\frac{\gd(x)^{\ga  }\gd(y)^{\ga  }}{\abs{x-y}^{2\ga  }}\right\}\right)\xd(y)^{(\ga-1)q   +1}|x-y|^{2-N}|y|^{-q(N+2\ga   -2)}\xf(y)dy.
\EAL $$
On the other hand
$$ \BAL
-\nabla  \xd(y)\nabla_y&\left(\min\left\{1,\gd(x)^{\ga  }\gd(y)^{\ga  }\abs{x-y}^{-2\ga  }\right\}\right)\\
&\leq 2\ga   |x-y|^{-1}\min\left\{1,\gd(x)^{\ga  }\gd(y)^{\ga  }\abs{x-y}^{-2\ga  }\right\} \text{ a.e. in } \xO.
\EAL $$
By collecting the above estimates, we obtain
\bel{in14} \BAL
-\int_{\xO_{\xb}} \nabla  \xd(y)\nabla F_\xm(x,y)&\xd(y)^{(\xa-1)q   +1}|y|^{-q(N+2\ga   -2)}\xf(y)dy\\
&\leq C\xd(x)^{\ga   }\int_\xO|x-y|^{-(N+\ga   -1)}|y|^{-q(N+\ga   -1)+1}dy.
\EAL \ee
It follows from  integration by parts, \eqref{in13} and \eqref{in14} that
\bel{le3} \BAL
&\int_{\xO_{\xb}} \xd(y)^{(\xa-1)q   }F_\xm(x,y)|y|^{-q(N+2\ga   -2)}\xf(y)dy\\
&=\frac{1}{(\xa-1)q   +1}\int_{\xO_{\xb}} \nabla(\xd(y)^{(\xa-1)q   +1})\nabla  \xd(y)F_\xm(x,y)|y|^{-q(N+2\ga   -2)}\xf(y)dy\\
&\leq C\xd(x)^{\ga   }\int_\xO|x-y|^{-(N+2\ga   -2)}|y|^{-q(N+\ga   -1)+\ga }dy\\
&+ C\xd(x)^{\ga   }\int_\xO|x-y|^{-(N+\ga   -1)}|y|^{-q(N+\ga   -1)+1}dy\\
&=:M(x)+N(x).
\EAL \ee

We will estimate $M(x)$ and $N(x)$ successively. By putting $e_x=|x|^{-1}x$ and $\eta=|x|^{-1}y$, we obtain
$$M(x)
\leq c\gd(x)^\ga |x|^{2-\ga -q(N+\ga   -1)}\int_{\BBR^N}|e_x-\eta|^{-(N+2\ga  -2)}|\eta|^{-q(N+\ga  -1)+\ga}d\eta.
$$
The last integral is finite and independent of $x$ since $q<q_\xm,$ $0<\ga   <1$ and $N\geq3$. This and \eqref{Kmu} imply that
\bel{le4} M(x) \leq C|x|^{N+\ga -q(N+\ga   -1)}K_{\xm}^\Gw(x,0).
\ee
Similarly we have
\bel{le5}\BAL N(x)
&\leq c|x|^{2-\ga-q(N+\ga   -1)}\int_{\BBR^N}|e_x-\eta|^{-(N+\ga  -1)}|\eta|^{-q(N+\ga  -1)+1}d\eta \\
&\leq C |x|^{N+\ga-q(N+\ga   -1)} K_{\xm}^\Gw(x,0).
\EAL \ee
Combining \eqref{ole1}, \eqref{le1}, \eqref{le2}, \eqref{le3} -- \eqref{le5}  implies that there exists a positive constant $C=C(\xO,N,\mu,q)>0$ such that
\begin{align}\label{1212}
\BBG_\mu^\Gw[|\nabla u|^q](x) &\leq Ck^q|x|^{N+\ga  -q(N+\ga   -1)} K_{\xm}^\Gw(x,0)\quad\forall x\in\xO.
\end{align}
By the above inequality and \eqref{formul} we can easily prove \eqref{BA.1}.

By combining \eqref{AB4} and \eqref{formul}, we obtain \eqref{BA.1}. The monotonicity comes from  Theorem B.
\end{proof}
\subsection{Strong singularities} We recall that $\CL_\mu$ is defined in \eqref{CLmu}. Notice that the eigenvalue $\gk_\mu$ is explicitly determined as follows
 \begin{equation}\label{Eq1-4}
 \gk_{\xm}=\ga(N+\ga-2)
 \end{equation}
and the corresponding eigenfunction
$\xf_\xm(\gs)=(\frac{x_N}{|x|}\lfloor_{S^{N-1}_+})^{\ga}=({\bf e}_{_N} \cdot \gs)^{\ga}$ solves
\begin{equation}\label{Eq7} \left\{ \BAL
-\CL_\mu \xf_\xm &=\gk_{\xm }\xf_\xm \quad&&\text{in }S^{N-1}_+
\\
\xf_\xm&=0 &&\text{on }\prt S^{N-1}_+.
\EAL \right. \end{equation}

Notice that equation $(\ref{Eq7})$ admits a unique positive solution with supremum $1$ and if $\xm=0$ then $\xa=1,$ which means that $\xf_0(\gs)={\bf e}_{_N} \cdot \gs$ is the first eigenfunction of $-\Gd'$ in $H^1_0(S^{N-1}_+)$. Moreover,
\begin{equation}\label{Eq4}
\gk_{\xm }=\displaystyle\inf\left\{
\frac{\int_{S^{N-1}_+}\left(|\nabla' w|^2-\xm ({\bf e}_{_N} \cdot \gs)^{-2}w^2\right)dS}{\int_{S^{N-1}_+}w^2dS}
:w\in H^1_0(S^{N-1}_+),w\neq 0\right\}.
\end{equation}
By \cite[Theorem 6.1]{dAdP} the infimum exists since $\xf_0(\xs)={\bf e}_{_N}\cdot \gs$ is the first eigenfunction of $-\Gd'$ in $H^1_0(S^{N-1}_+)$. The minimizer $\gf_\xm$ belongs to $H^1_0(S^{N-1}_+)$ only if $1<\mu<\frac{1}{4}$. Furthermore $\gf_{\gm}\in {\bf Y}_\mu(S^{N-1}_+)$ where  ${\bf Y}_\mu(S^{N-1}_+)$ is defined in \eqref{Y}. 

Finally by \eqref{Eq7} the following expression holds
\be
|\nabla'\xf_0(\xs)|^2=1-\xf_0(\xs)^2 \quad\forall \xs \in S^{N-1}_+.\label{gradientf}
\ee
\noindent Indeed, since $\xf_{\frac{1}{4}}=\xf_0^{\frac{1}{2}}$ we have

$$-\xD'\xf_0^{\frac{1}{2}}=\frac{1}{4}\xf_0^{-\frac{3}{2}}|\nabla'\xf_0|^2+\frac{N-1}{2}\xf_0^\frac{1}{2}$$
and
$$-\xD'\xf_0^{\frac{1}{2}}=\frac{1}{4}\xf_0^{-\frac{3}{2}}+\gk_{\frac{1}{4}}\xf_0^\frac{1}{2}.$$

\noindent Taking into account that  $\gk_{\frac{1}{4}}-\frac{N-1}{2}=-\frac{1}{4},$ by the above equalities we obtain \eqref{gradientf}.

\begin{proof}[{\bf Proof of Theorem G}]
\textbf{Step 1.} \emph{Existence.} Set $$\xg_1:=\left(\frac{\ell_{N,q}-\gk_\xm}{\xa^q}\right)^{\frac{1}{q-1}},$$ then the function $\overline{\xo}:=\xg_1\xf_\xm$ is a supersolution of \eqref{P4}.
Indeed by \eqref{Eq7} and \eqref{gradientf},
\begin{align*}
-\CL_\mu \overline{\gw} - \ell_{N,q}\overline{\gw} &+J(\overline \gw, \nabla' \overline \gw)
=\xg_1(\gk_\xm- \ell_{N,q})\xf_\xm+ \xg_1^q\left( \Big(\frac{2-q}{q-1}\Big)^2 \xf_\xm^2 + \abs{\nabla' \xf_\xm}^2\right)^{\frac{q}{2}} \\
&=\xg_1(\gk_\xm- \ell_{N,q})\xf_0^\xa+ \xg_1^q\left( \Big(\frac{2-q}{q-1}\Big)^2 \xf_0^{2\xa} +\xa^2\xf_0^{2(\xa-1)} \abs{\nabla' \xf_0}^2\right)^{\frac{q}{2}}\\
&=\xg_1(\gk_\xm- \ell_{N,q})\xf_0^\xa+ \xg_1^q\left( \left(\Big(\frac{2-q}{q-1}\Big)^2-\xa^2\right) \xf_0^{2\xa} +\xa^2\xf_0^{2(\xa-1)}\right)^{\frac{q}{2}}\\
&\geq \xg_1(\gk_\xm- \ell_{N,q})\xf_0^\xa+\xa^q\xg_1^q\xf_0^{q(\xa-1)}\\
&\geq [\xg_1(\gk_\xm- \ell_{N,q})+\xa^q\xg_1^q]\gf_0^\ga=0
\end{align*}

Let $\xa_0\in (\xa,1)$ be such that
$$q<\frac{N+\xa_0}{N+\xa_0-1}<q_\mu.$$

We note that $\xf_{\xm_0}=\xf_0^{\xa_0},$ where  $$\xm_0=\frac{1}{4}-(\xa_0-\frac{1}{2})^2<\mu.$$

We allege that there exists a positive constant $\xg_2=\xg_2(N,q,\xm,\xm_0)\leq \xg_1$ such that the function $\underline{\xo}=\xg_2\xf_{\xm_0}$ is a subsolution of \eqref{P4}.
Indeed by \eqref{Eq7} and \eqref{gradientf} we have
\begin{align*}
&-\CL_\mu \underline{\xo} - \ell_{N,q}\underline{\xo} + J(\underline \gw, \nabla' \underline \gw)\\
&=\xg_2(\xm_0-\xm)\frac{\xf_{\xm_0}}{({\bf e}_{_N} \cdot \gs)^2}+\xg_2(\gk_{\xm_0}-\ell_{N,q})\xf_{\xm_0}+\xg_2^q\left( \Big(\frac{2-q}{q-1}\Big)^2 \xf_{\xm_0}^2 + \abs{\nabla' \xf_{\xm_0}}^2\right)^{\frac{q}{2}}\\
&=\xg_2(\xm_0-\xm)\xf_{0}^{\xa_0-2}+\xg_2(\gk_{\xm_0}-\ell_{N,q})\xf_{0}^{\xa_0}+\xg_2^q\left( \left(\Big(\frac{2-q}{q-1}\Big)^2-\xa_0^2\right) \xf_0^{2\xa_0} +\xa_0^2\xf_0^{2(\xa_0-1)}\right)^{\frac{q}{2}}\\
&\leq\left(\xg_2(\xm_0-\xm)+\xg_2^q\xa_0^q\right)\xf_{0}^{\xa_0-2}+\left(\xg_2(\gk_{\xm_0}-\ell_{N,q})
+\xg_2^q\left(\Big(\frac{2-q}{q-1}\Big)^2-\xa_0^2\right)^{\frac{q}{2}}\right)\xf_{0}^{\xa_0}\leq0,
\end{align*}
provided $\xg_2$ is small enough. Notice that we can choose $\xg_2\leq \xg_1$.

For $t\in (0,1)$, set $S_t:=\{\gs \in S_+^{N-1}: \xf_0(\gs)<t\}$ and $\widetilde S_t:=S_+^{N-1} \setminus S_t$. In view of the proof of \cite[Theorem 6.5]{Ka}, there exists a solution $\xo_t \in W^{2,p}(\widetilde{S}_t)$ to \eqref{P4} such that
\be
\underline{\xo}(\xs)\leq \xo_t(\xs)\leq \overline{\xo}(\xs) \quad\forall\xs\in\;\widetilde{S}_t .\label{supest}
\ee
Therefore, by the standard elliptic theory, there exist a function $\widetilde w$ and a sequence $t_n\searrow0$  such that  $\xo_{t_n}\rightarrow\widetilde{\xo}$ locally uniformly in $C^1(S^{N-1}_+)$ and $\widetilde{\xo}$ satisfies
 $$-\CL_\mu \widetilde{\gw} - \ell_{N,q}\widetilde{\gw} + J(\widetilde \gw,\nabla' \widetilde \gw) = 0 \quad\text{in } S_+^{N-1}.$$
Furthermore by \eqref{supest} we have that
\be
\underline{\xo}(\xs)\leq \widetilde{\xo}(\xs)\leq \overline{\xo}(\xs) \quad\forall\xs\in\; S_+^{N-1}.\label{supest1}
\ee

Set $\tilde{u}(x)=|x|^{-\frac{2-q}{q-1}}\widetilde{\xo}(\gs)$ then $\tilde{u}$ satisfies
$$- L_\mu \tilde{u} +|\nabla \tilde{u}|^q = 0 \quad \text{in } \mathbb{R}^N_+,$$
and
$$|\tilde{u}(x)|\leq \left(\frac{\ell_{N,q}-\gk_\xm}{\xa^q}\right)^{\frac{1}{q-1}} x_N^\xa|x|^{-\frac{2-q}{q-1}+\xa} \quad\forall x\in \mathbb{R}^N_+.$$

Let $x_0=(x_0',0)$ be such that $|x_0'|=1$ then in view of the proof of \eqref{apriori2}, there exists a constant $C_1=C(N,\mu,q)$ such that
\be
|\nabla\tilde{u}(x)|\leq C_1x_N^{\xa-1} \quad\forall x\in B(x_0,\frac{1}{2}).\label{kk}
\ee
This implies
\be
|\nabla'\tilde{\xo}(\xs)|\leq C \xf_0(\gs)^{\xa-1} \quad\forall \xs\in S^{N-1}_+.\label{grest}
\ee

\noindent \textbf{Step 2.} \emph{Uniqueness.}

Let $\xo_i\in {\bf Y}_\mu(S^{N-1}_+)$, $i=1,2,$ be two positive solutions of \eqref{P4}.
Let $x_0=(x_0',0)$ be such that $|x_0'|=1,$ Then $u_i(x)=|x|^{-\frac{2-q}{q-1}}\xo_i\in H^1(B(x_0,\frac{2}{3}),x_N^{2\xa})$ and satisfies
$$- L_\mu u_i +|\nabla u_i|^q = 0 \quad \text{in } \mathbb{R}^N_+,$$
which implies $- L_\mu u_i \leq 0$  in  $\mathbb{R}^N_+$. Since $0<v_i:=x_N^{-\ga}u_i \in H^1(B(x_0,\frac{1}{2}),\;x_N^{2\xa})$ and satisfies
$$-\div(x_N^{2\ga} \nabla v)\leq 0 \quad \text{in } \mathbb{R}^N_+,$$
by \cite[Theorem 2.12]{FMT}, there exists a positive constant $C_i>0$ such that
$$
u_i(x) \leq C_i x_N^{\xa} \quad\forall x\in B(x_0,\frac{1}{2}).
$$
Therefore in view of the proof of \eqref{grest}, we deduce that there exists a positive constant $C_0$ such that
\be
 w_i(\xs)\leq C_0 \xf_0(\xs)^\ga \quad\forall \xs\in S^{N-1}_+,\;\;i=1,2\label{boundsf}
\ee
and
\be
 |\nabla'w_i(\xs)|\leq C_0 \xf_0(\gs)^{\xa-1} \quad\forall \xs\in S^{N-1}_+,\;\;i=1,2\label{boundsgrandf}.
\ee

Set
$$b_t:=\inf_{c>1}\{c: c\xo_1\geq\xo_2,\xs\in \widetilde S_t\}<\infty.$$
Without loss of generality we may assume that $b_{t_0}>1$ for some $t_0\in(0,1)$; thus by \eqref{boundsf} we have
$$1<b_{t_0}\leq b_t \quad\forall t \in (0,t_0).$$

In the sequel we consider $t \in (0,t_0)$. Put $\psi:=\xf_0^\xa-\frac{1}{2}\xf_0^{\xa+\xe},$ where $\xe\in (0,1-\xa)$ is a parameter that will be determined later. Then we have
$\frac{1}{2}\xf_0^\xa\leq\psi\leq \xf_0^\xa.$ We recall that $\xf_0^\xa=\xf_\xm$ and $\xf_0^{\xa+\xe}=\xf_{\xm_\xe}$ where
$$\xm_\xe:=\frac{1}{4}-(\xa+\xe-\frac{1}{2})^2.$$
From the definition of $\psi$, it is easy to check that
\bel{psi}
-\CL_\mu \psi =\frac{\xm-\xm_\xe}{2}\xf_0^{\xa+\xe-2}+\xf_0^{\xa}\left(\gk_\xm-\frac{\gk_{\xm_\xe}}{2}\xf_0^{\xe}\right).
\ee

Now, let $\xo_t=b_t^{-1}\xo_2.$ We remark that $\xo_t$  is a subsolution of \eqref{P4} and $\xo_t-\xo_{1}\leq 0$ in $\widetilde S_t$. Also, we have
\bel{dior1}
-\CL_\mu (\xo_t-\xo_1) \leq -J(\gw_t,\nabla' \gw_t) + J(\gw_1,\nabla' \gw_1)
+\ell_{N,q}(\xo_t -\xo_{1}).
\ee

Since $1<q<2$, the following inequality holds for any nonnegative number $h_1,h_2,k_1,k_2$
 \bel{alge}
 -(h_1^2+h_2^2)^\frac{q}{2}+(k_1^2+k_2^2)^\frac{q}{2}
 \leq \left(h_1^{q-1}+h_2^{q-1}+k_1^{q-1}+k_2^{q-1}\right)\left(|h_1-k_1|+|h_2-k_2|\right).
 \ee
By applying \eqref{alge} with $h_1=\Big(\frac{2-q}{q-1}\Big) \gw_\xd$, $h_2=\abs{\nabla' \gw_\xd}$, $k_1=\Big(\frac{2-q}{q-1}\Big)\gw_1$ and $k_2=\abs{\nabla' \gw_1}$ and keeping in mind estimates \eqref{boundsf} and \eqref{boundsgrandf}, we obtain
 \bel{dior}
 -J(\gw_t, \nabla' \gw_t) + J(\gw_1,\nabla' \gw_1)\\
 \leq C(q,C_0)\xf_0^{(q-1)(\xa-1)}\left(|\xo_t-\xo_1|+|\nabla'(\xo_t-\xo_1)|\right).
 \ee

 Now set $V_t:=\psi^{-1}(\xo_t-\xo_1)$.  By \eqref{dior1}, \eqref{dior} and the definition of $\psi,$ we can easily deduce the existence of a positive constant $C=C(N,\mu,q,C_0)$ such that
\bel{trem}
-\text{div}'(\psi^2\nabla'V_t) + \psi V_t(-\CL_\mu\psi)\leq C\left(\xf_0^{q(\xa-1)+\xa}|V_t|+\xf_0^{(q-1)(\xa-1)+2\xa}|\nabla'V_t|\right).
\ee
Now since $\psi V_t \in{\bf Y}_\mu(S^{N-1}_+)$ and $V_t(\xs)\leq0$ for any $\xs \in \widetilde S_t,$ multiplying the above inequality by $(V_t)_+$ and integrating over $S_+^{N-1}$, we get
\bel{ineq1} \BAL
\int_{S_t} &|\nabla'(V_t)_+|^2\psi^2 dS(\gs)+\int_{S_t} \psi (V_t)_+^2(-\CL_\mu\psi) dS(\gs)=\psi (-\CL_\mu (\gw_t-\gw_1)) \\
&\leq C\left( \int_{S_t} \xf_0^{q(\xa-1)+\xa}(V_t)_+^{2} dS(\gs) +\int_{S_t}\xf_0^{(q-1)(\xa-1)+2\xa}|\nabla'(V_t)_+|(V_t)_+ dS(\gs)\right).
\EAL \ee
By the definition of $\psi$  and \eqref{psi}, we have
\bel{ineq2}\BAL
\int_{S_t} |\nabla'(V_t)_+|^2&\psi^2 dS(\gs)+ \int_{S_t} \psi (V_t)_+^2(-\CL_\mu\psi) dS(\gs)\\
&\geq \frac{1}{4}\int_{S_t} |\nabla'(V_t)_+|^2\xf_0^{2\xa} dS(\gs)
+\frac{\xm-\xm_\xe}{4}\int_{S_t}  (V_t)_+^2\xf_0^{2\xa+\xe-2}dS(\gs)\\
&-\frac{N-1}{2}\int_{S_t}  (V_t)_+^2\xf_0^{2\xa}dS(\gs).
\EAL \ee

\noindent Note here that if $\xe<1-\xa$ then $q<2<\frac{2-\xa-\xe}{1-\xa}$. This leads to
\bel{gamma}
2-\ga-\vge-q(1-\ga)>0 \quad \text{and} \quad  4-2\ga-\vge-2q(1-\ga)>0.
\ee

By Young's inequality, we deduce that
\bel{ineq4}\BAL
C\int_{S_t}&\xf_0^{(q-1)(\xa-1)+2\xa}|\nabla'(V_t)_+|(V_t)_+ dS(\gs)\\
&\leq \frac{1}{8}\int_{S_t}\xf_0^{2\xa}|\nabla'(V_t)_+|^2 dS(\gs)+ \hat C\int_{S_t}\xf_0^{2(q-1)(\xa-1)+2\xa}(V_t)_+^2 dS(\gs)
\EAL \ee
where $C$ is the constant in \eqref{ineq1} and $\hat C=\hat C(q,N,\mu)$.

Gathering \eqref{ineq1}, \eqref{ineq2} and \eqref{ineq4} yields
$$\BAL
\frac{1}{8}\int_{S_t}&\xf_0^{2\xa}|\nabla'(V_t)_+|^2dS(\gs)\leq-\frac{\xm-\xm_\xe}{4}\int_{S_t} \xf_0^{2\xa+\xe-2} (V_t)_+^2dS(\gs) \\
&\;\;\quad+C_1\int_{S_t}\left(\xf_0^{q(\xa-1)+\xa} + \xf_0^{2(q-1)(\xa-1)+2\xa}\right) (V_t)_+^2dS(\gs)\\ \nonumber
&\leq \int_{S_t}\xf_0^{2\xa+\xe-2}\left(-\frac{\xm-\xm_\xe}{4}+C_1\left(t^{2-\ga-\vge-q(1-\ga)}  + t^{4-2\ga-\vge-2q(1-\ga)} +t^{2-\xe}\right) \right)(V_t)_+^2dS(\gs),
\EAL $$
where $C_1=C(q,N,\xm)$. By \eqref{gamma} and the above inequality we can find a positive constant $t_1=t_1(N,q,\mu,\xe,C_0)$ such that
$$\frac{1}{8}\int_{S_{t_1}}\xf_0^{2\xa}|\nabla'(V_{t_1})_+|^2dS(\gs)\leq 0,$$
which implies $(V_{t_1})_+=0$ in $S_{t_1}$ since $(V_{t_1})_+=0$ on $\{\gs \in S_+^{N-1}: \xf_0(\gs)=t_1\}.$ Hence
\be
b_{t_1}^{-1}\xo_2\leq \xo_1 \quad \forall \xs \in S_{t_1}.
\ee

Thus we have proved that $$b_{t_1}=\inf_{c>1}\{c: c\xo_1\geq\xo_2,\xs\in \widetilde S_{t_1}\}=\inf_{c>1}\{c: c\xo_1\geq\xo_2,\xs\in S^{N-1}_+\}.$$
This means that $\gw_1 - \xo_{t_1}(\xs)\geq0$ for any $\xs\in S^{N-1}_+$ and
\be
\xo_1(\xs_0)-\gw_{t_1}(\xs_0)=0,\quad\text{for some}\;\xs_0\in \widetilde{S}_{t_1}.\label{contr}
\ee
But
\begin{align*}
-\CL_\mu (\xo_1-\gw_{t_1}) &- \ell_{N,q}(\xo_1-\gw_{t_1}) + J(\gw_1,\nabla' \gw_1) - J(\gw_{t_1},\nabla' \gw_{t_1}) \geq 0,
\end{align*}
which implies
\begin{align*}
-\xD'(\xo_{1}-\gw_{t_1}) +H(\xo_{1},\nabla' \gw_{1})-H(\xo_{t_1},\nabla' \gw_{t_1}) \geq 0.
\end{align*}
By the above inequality and mean value theorem there exists $\bar \Gl>0$ such that
$$-\xD'(\xo_{1}-\gw_{t_1}) +\frac{\partial J(\overline{s},\overline{\xi})}{\partial\xi}(\nabla' \gw_1-\nabla' \gw_{t_1})+\bar \Gl(\xo_{1}-\gw_{t_1}) \geq 0 \quad \text{in } \widetilde S_{\frac{t_1}{2}}, $$
where $\overline{s}$ and $\overline{\xi}$ are  functions with respect to $\xs\in \widetilde S_{\frac{t_1}{2}} $ such that $\frac{\partial H(\overline{s},\overline{\xi})}{\partial\xi}\in L^\infty(\widetilde S_{\frac{t_1}{2}}).$ By the maximum principle, $\xo_{1}-\gw_{t_1}$
cannot achieve a non-positive minimum in $\widetilde S_{\frac{t_1}{2}} \setminus \partial \widetilde S_{\frac{t_1}{2}}$ which clearly contradicts \eqref{contr}.

The result follows by exchanging the role of $\xo_1,\;\xo_2.$
\end{proof}

By Theorem F and Proposition E, the sequence $\{u_{0,k}^\Gw\}$ is increasing and bounded from above, hence there exists $u_{0,\infty}^\Gw=\lim_{k \to \infty}u_{0,k}^\Gw$.

\begin{proof}[{\bf Proof of Theorem H}]
	
By Proposition E, for every $k>0$,
\bel{ukright} u_{0,k}^\Gw(x) \leq c \gd(x)^\ga |x|^{-\frac{2-q}{q-1}-\ga} \quad \forall x \in \Gw, \ee
and
\bel{gradinf11}
|\nabla u_{0,k}^\Gw(x)| \leq c\gd(x)^{\ga-1}|x|^{-\frac{2-q}{q-1}-\ga} \quad \forall x \in \Gw,
\ee
where $c=c(N,\mu,q,\Gw)$.
Moreover $\{u_{0,k}^\Gw\}$ is increasing. Therefore, by the standard regularity result,  $u_{0,k}^\Gw \to u_{0,\infty}^\Gw$ in $C_{loc}^1(\Gw)$ and $u_{0,\infty}^\Gw$ is a positive solution of \eqref{Epower}. Letting $k \to \infty$ in \eqref{ukright} and \eqref{gradinf11} yields the second estimate in \eqref{uinf2side} and estimate \eqref{gradinf}.

Next we infer from \eqref{formul}, \eqref{1212} and \eqref{Kmu} that, for every $k>0$,
$$ \BAL u_{0,k}^\Gw(x) &= k K_\mu^\Gw(x,0) - \BBG_\mu^\Gw[|\nabla u_{0,k}^\Gw|^q](x) \\
&\geq k K_\mu^\Gw(x,0) - c_1 k^q |x|^{N+\ga - (N+\ga-1)q}K_\mu^\Gw(x,0) \\
&\geq c_2k\gd(x)^\ga |x|^{2-N-2\ga}(1- c_3k^{q-1} |x|^{N+\ga - (N+\ga-1)q}  ).
\EAL $$
For $x \in \Gw$,  choose $k=a|x|^{-\frac{N+\ga-(N+\ga-1)q}{q-1}}$, where $a>0$ will be made precise later on, then
$$ u_{0,k}^\Gw(x) \geq c_2 a \, \gd(x)^\xa |x|^{-\frac{2-q}{q-1}-\ga}(1-c_3a^{q-1}).
$$
By choosing $a=(2c_3)^{-\frac{1}{q-1}}$, we deduce for any $x \in \Gw$ there exists $k>0$ depending on $|x|$ such that
$$ u_{0,k}^\Gw(x) \geq c_4 \gd(x)^\xa|x|^{-\frac{2-q}{q-1}-\ga}. $$
Since $u_{0,\infty}^\Gw \geq u_{0,k}^\Gw$ in $\Gw$ we obtain the first inequality in \eqref{uinf2side}.

Next we prove \eqref{uniq7}. Since $u_{0,k}^\Gw$ is the solution of \eqref{Pc}, it follows that, for any $\ell>0$,  $T_\ell[u_{0,k}^\Gw]$ is a solution of
	\bel{Pkl} \left\{ \BAL -L_\mu u + \abs{\nabla u}^q  &= 0 \quad \text{in } \Gw^\ell \\
	\tr(u) &= k \ell^{m_q} \gd_0.
	\EAL \right. \ee
where $m_q=\frac{2-q}{q-1} - N+1$ and $T_\ell$ is given in \eqref{scal}. By the uniqueness,
\bel{ui1} T_\ell[u_{0,k}^\Gw]=u_{0, k \ell^{m_q}}^{\Gw^\ell} \quad \text{ in } \Gw^\ell. \ee
Sending $k \to \infty$ implies  $T_\ell[u_{0,\infty}^\Gw]=u_{0,\infty}^{\Gw^\ell}$ in $\Gw^\ell$.

 Since estimates \eqref{uinf2side} and \eqref{gradinf11} are invariant under the transformation $T_\ell$, it follows that
\bel{uq1}
c^{-1}\dist(x,\prt \Gw^\ell)^\ga |x|^{-\frac{2-q}{q-1}-\ga} \leq u_{0,\infty}^{\Gw^\ell}(x) \leq c\, \dist(x,\prt \Gw^\ell)^\ga |x|^{-\frac{2-q}{q-1}-\ga} \quad \forall x \in \Gw^\ell,
\ee
\bel{graduq2}
|\nabla u_{0,\infty}^{\Gw^\ell}(x)| \leq c\,\dist(x,\prt \Gw^\ell)^{\ga-1}|x|^{-\frac{2-q}{q-1}-\ga} \quad \forall x \in \Gw^\ell.
\ee
By local regularity results, we deduce that there exist a function $U$ and a subsequence $\{ \ell_n \}$ such that $u_{0,\infty}^{\Gw^{\ell_n}} \to U$, as $\ell_n \to 0$, in $C_{loc}^1(\BBR_+^N)$. Furthermore, $U$ is positive solution of \eqref{P3}.

From \eqref{ui1}, for any $\ell'>0$, we have
$$ T_{\ell'}[u_{0,k\ell_n^{m_q}}^{\Gw^{\ell_n}}]= u_{0,k(\ell_n \ell')^{m_q}}^{\Gw^{\ell_n \ell'} } \quad \text{in } \Gw^{\ell_n \ell'}.$$
By letting $k \to \infty$ and $\ell_n \to 0$, we obtain that $T_{\ell'}[U]=U$ in $\BBR_+^N$ for every $\ell'$. By putting $\gw(\gs)= r^{\frac{2-q}{q-1}}U(x)$ with $r=|x|$, $\gs=r^{-1} x$, we deduce that $\gw$ is a positive solution of \eqref{P4}. By Theorem G, $\gw$ is the unique solution of \eqref{P4}. Thus \eqref{uniq7} follows.
\end{proof}
\section{Supercritical Case}
We start the section with an observation that when $g(t)=t^q$ the condition \eqref{G1} is fulfilled if and only if $q<q_\xm$; in which case the solvability of \eqref{N1} holds for every $\mu \in \GTM^+(\prt \Gw)$.  On the contrary, in the {\it supercritical case} i.e. if $q\geq q_\xm$, a continuity condition with respect to some Besov capacity is needed to derive an existence result.

We recall below some notations concerning Besov space (see, e.g., \cite{Ad, Stein}). For $\gs>0$, $1\leq p<\infty$, we denote by $W^{\gs,p}(\BBR^d)$ the Sobolev space over $\BBR^d$. If $\gs$ is not an integer the Besov space $B^{\gs,p}(\BBR^d)$ coincides with $W^{\gs,p}(\BBR^d)$. When $\gs$ is an integer we denote $\Gd_{x,y}f:=f(x+y)+f(x-y)-2f(x)$ and
$$B^{1,p}(\BBR^d):=\left\{f\in L^p(\BBR^d): \frac{\Gd_{x,y}f}{|y|^{1+\frac{d}{p}}}\in L^p(\BBR^d\times \BBR^d)\right\},
$$
with norm
$$\|f\|_{B^{1,p}}:=\left(\|f\|^p_{L^p}+\int\int_{\BBR^d\times \BBR^d}\frac{|\Gd_{x,y}f|^p}{|y|^{p+d}}dxdy\right)^{\frac{1}{p}}.
$$
Then
$$B^{m,p}(\BBR^d):=\left\{f\in W^{m-1,p}(\BBR^d): D_x^\ga f\in B^{1,p}(\BBR^d)\;\forall\ga\in \BBN^d\,|\ga|=m-1\right\},
$$
with norm
$$\|f\|_{B^{m,p}}:=\left(\|f\|^p_{W^{m-1,p}}+\sum_{|\ga|=m-1}\int\int_{\BBR^d\times \BBR^d}\frac{|D_x^\ga\Gd_{x,y}f|^p}{|y|^{p+d}}dxdy\right)^{\frac{1}{p}}.
$$
These spaces are fundamental because they are stable under the real interpolation method  developed by Lions and Petree.
For $\ga\in\BBR$ we defined the Bessel kernel of order $\ga$ by $G_\ga(\xi)=\CF^{-1}(1+|.|^2)^{-\frac{\ga}{2}}\CF(\xi)$, where $\CF$ is the Fourier transform of moderate distributions in $\BBR^d$. The Bessel space $L_{\ga,p}(\BBR^d)$ is defined by
$$L_{\ga,p}(\BBR^d):=\{f=G_\ga\ast g:g\in L^{p}(\BBR^d)\},
$$
with norm
$$\|f\|_{L_{\ga,p}}:=\|g\|_{L^p}=\|G_{-\ga}\ast f\|_{L^p}.
$$
It is known that if $1<p<\infty$ and $\ga>0$, $L_{\ga,p}(\BBR^d)=W^{\ga,p}(\BBR^d)$ if $\ga\in\BBN$ and $L_{\ga,p}(\BBR^d)=B^{\ga,p}(\BBR^d)$ if $\ga\notin\BBN$, always with equivalent norms. The Bessel capacity is defined for compact subsets
$K\subset\BBR^d$ by
$$C^{\BBR^d}_{\ga,p}(K):=\inf\{\|f\|^p_{L_{\ga,p}}, f\in\CS'(\BBR^d),\,f\geq \chi_K\}.
$$
It is extended to open sets and then Borel sets by the fact that it is an outer measure.

Let us recall the following result in \cite{MaNg}.
\begin{proposition} \label{fragma}
Let $\xn\in\mathfrak M^+(\prt\Gw)$ and $\xb_0$ be the constant in Proposition \eqref{geo}. Then the following inequalities hold
\begin{align}\label{fragma1}
\sup_{0<\xb\leq\xb_0}\xb^{\xa-1}\int_{\Gs_\xb}\BBK_\mu^\xO[\nu]dS&\leq C(\xb_0,\xa,\xO)||\nu||_{\mathfrak M(\prt\Gw)},\quad\text{if}\;\xm<\frac{1}{4},\\ \label{fragma2}
\sup_{0<\xb\leq\xb_0}\left(\xb|\log\xb|^2\right)^{-\frac{1}{2}}\int_{\Gs_\xb}\BBK_\mu^\xO[\nu]dS&\leq C(\xb_0,\xa,\xO)||\nu||_{\mathfrak M(\prt\Gw)},\quad\text{if}\;\xm=\frac{1}{4}.
\end{align}
\end{proposition}
\begin{proof} Estimates \eqref{fragma1} follows from \cite[Corollary 2.10]{MaNg}. Estimate \eqref{fragma2} can 
be obtained by a similar argument with some modifications and we omit it.
\end{proof}

\begin{lemma}
Let $\xn\in\mathfrak M^+(\prt\Gw),$ $q\in(1,2)$  and $u\in C^2(\xO)$ be a nonnegative solution of \eqref{N1}.

(i) If $q\neq\xa+1$ then there exists a constant $\xb_1=\xb_1(N,\mu,q,\xO)>0$ such that the following inequality holds
\be\label{aniso1}
\int_{\xO}\gd^{\ga-q}u^q dx\leq C\left(\int_{\xO}\xd^\xa|\nabla u|^qdx+1\right),
\ee
where $C$ depends only on  $N$, $\mu$, $q$, $\Gw$ and $\sup_{ \Gs_{\xb_1}}(\BBK_\mu^\xO[\nu])^q.$

(ii) If $q=\xa+1$ then for any $\vge>0$ small enough there exists a constant $\xb_1=\xb_1(N,\mu,\xO,\vge)>0$ such that the following inequality holds
\be\label{aniso1b}
\int_{\xO}\gd^{\vge-1} u^{\xa+1} dx\leq C\left(\int_{\xO}\xd^\xa|\nabla u|^{\xa+1}dx+1\right),
\ee
where $C$ depends only on $N$, $\mu$, $\Gw$, $\vge$ and $\sup_{\Gs_{\xb_1}}(\BBK_\mu^\xO[\nu])^{\xa+1}.$
\end{lemma}
\begin{proof}
Since $u$ is a nonnegative solution of \eqref{N1} we have that $|\nabla u|\in L^q(\xO,\xd^\xa)$. Let $\xb_1 \in (0,\xb_0)$ where $\xb_0$ is the constant in Proposition \ref{geo}.

(i) First we assume that $q \neq \xa+1$ and let $\gg \neq -1$.
Then for $\gb \in (0,\gb_1)$,
\begin{align*}
\int_{D_\xb\setminus D_{\xb_1}}\gd^\gg  u^q dx&=(\gg+1)^{-1}\int_{D_\xb\setminus D_{\xb_1}}\nabla \xd^{\gg+1}\nabla \xd \, u^qdx\\
&=(\gg+1)^{-1}\left(-\int_{D_\xb\setminus D_{\xb_1}} \xd^{\gg+1} \xD\xd u^qdx-q\int_{D_\xb\setminus D_{\xb_1}} \xd^{\gg+1}u^{q-1} \nabla\xd\nabla udx\right.\\
&\left.\;\;\;+\int_{\Gs_{\xb_1}} \xd^{\gg+1} \frac{\partial\xd}{\partial \mathbf{n}_{\xb_1}} u^qdx+\int_{\Gs_{\xb}} \xd^{\gg+1} \frac{\partial\xd}{\partial \mathbf{n}_{\xb}} u^qdx\right)\\
&\leq C|\gg+1|^{-1}\left(\int_{D_\xb\setminus D_{\xb_1}} \xd^{\gg+1}u^qdx+\int_{D_\xb\setminus D_{\xb_1}} \xd^{\gg+1} u^{q-1} |\nabla u|dx\right.\\
&\left.\;\;\;+ \xb_1^{\gg+1}\sup_{ \Gs_{\xb_1}}(\BBK_\mu^\xO[\nu])^q+\int_{\Gs_{\xb}} \xd^{\gg+1}u^qdx\right).
\end{align*}
Observe that for any $\gg \in (\ga-q, \max\{ \frac{\ga-1-q}{2}, 2(\ga-q) +1  \})$, we have
\bel{notzero} |\gg+1|^{-1} < 2|\ga+1-q|^{-1}.\ee
Therefore, for such $\gg$, we can choose $\xb_1=\xb_1(N,q,\mu,\Gw)$ such that
$$C|\gg+1|^{-1}\int_{D_\xb\setminus D_{\xb_1}} \xd^{\gg+1} u^qdx\leq 2C|\ga+1-q|^{-1}\int_{D_\xb\setminus D_{\xb_1}} \xd^{\gg+1} u^qdx  \leq\frac{1}{4}\int_{D_\xb\setminus D_{\xb_1}} \xd^{\gg}u^qdx.$$
Consequently, by H\"{older} inequality we can find a constant $C_1=C_1(N,q,\mu,\Gw)$ such that
\begin{align*}
C|\gg+1|^{-1}\int_{D_\xb\setminus D_{\xb_1}} \xd^{\gg+1}u^{q-1} |\nabla u|dx\leq \frac{1}{4}\int_{D_\xb\setminus D_{\xb_1}} \xd^{\gg}u^qdx+
C_1\int_{D_\xb\setminus D_{\xb_1}} \xd^{\gg+q}|\nabla u|^qdx.
\end{align*}
By the above estimates, there is a positive constant $C_2=C_2(N,\mu,q,\Gw)$ such that
\begin{align} \label{anis2}
\int_{D_\xb\setminus D_{\xb_1}}\gd^{\gg} u^q dx\leq C_2\left(\int_{D_\xb\setminus D_{\xb_1}} \xd^{\gg+q}|\nabla u|^qdx\right.&+\left.\xb_1^{\gg+1}\sup_{\Gs_{\xb_1}}(\BBK_\mu^\xO[\nu])^q+\int_{\Gs_{\xb}} \xd^{\gg+1}u^qdx\right).
\end{align}
By \eqref{apriori1}, Proposition \ref{fragma} and taking into account that $\gg+q-1>\ga-1$, we obtain
$$\int_{\Gs_{\xb}} \xd^{\gg+1}u^qdS\leq C\xb^{\gg+q-1}\int_{\Gs_{\xb}}udS\leq C\xb^{\gg+q-1}\int_{\Gs_{\xb}}\BBK_\mu^\xO[\nu]dS\to 0 \quad \text{as } \gb \to 0.$$
Therefore, by letting $\gb \to 0$ in \eqref{anis2}, we obtain
\bel{anis3}
\int_{\xO_{\xb_1}} \gd^{\gg}u^q dx\leq C_2\left(\int_{\xO_{\xb_1}} \xd^{\gg+q}|\nabla u|^qdx+\xb_1^{\gg+1}\sup_{\Gs_{\xb_1}}(\BBK_\mu^\xO[\nu])^q\right).
\ee
By dominated convergence theorem, we can send $\gg \to \ga-q$ in \eqref{anis3} to obtain
\bel{anis4}
\int_{\xO_{\xb_1}} \gd^{\ga-q}u^q dx\leq C_2\left(\int_{\xO_{\xb_1}} \xd^{\ga}|\nabla u|^qdx+\xb_1^{\gg+1}\sup_{\Gs_{\xb_1}}(\BBK_\mu^\xO[\nu])^q\right).
\ee
This implies \eqref{aniso1}.

The proof of \eqref{aniso1b} follows by similar arguments as the proof of \eqref{aniso1} (with $\gg=\vge-1$) with the some modifications and we omit it.
\end{proof}

Next put $\Gs=\prt \Gw$ and denote by $\Gd_\Gs$ the Laplace-Beltrami operator on $\Gs$.

\begin{proof}[{\bf Proof of Theorem I}]
Let $\xe\geq0$ and $u$ be the solution of \eqref{N1}. If $\eta\in L^\infty(\Gs) \cap B^{\xe+\frac{\xa+1}{q}-\xa,q'}(\Gs)$, we denote by $H:=H[\eta]$ the solution of
 \begin{equation}\label{N12} \left\{  \BAL
\frac{\prt H}{\prt s}+\Gd_\Gs H &=0\qquad&&\text{in }(0,\infty)\times\Gs,\\
H(0,.)&=\eta\qquad&&\text{on }\Gs.
\EAL \right.  \end{equation}
Let $h\in C^\infty(\BBR_+)$ such that $0\leq h\leq 1$, $h'\leq 0$, $h\equiv 1$ on $[0,\frac{\gb_0}{2}]$, $h\equiv 0$ on
$[\gb_0,\infty]$. The lifting we consider is expressed by
  \begin{equation}\label{N13}
R[\eta](x):=\left\{\BA {lll}H[\eta](\gd^2,\gs(x))h(\gd)\qquad&\text{if }x\in \overline\Gw_{\gb_0}\\
0&\text{if }x\in D_{\gb_0},
\EA\right.
\end{equation}
with $x\approx (\gd,\gs)=(\gd(x),\gs(x))$.

\noindent \textbf{Case 1:} $q\neq\xa+1.$ Set $\xe=0$ and $\gz=\vgf_\gm R[\eta]^{q'}$ where $\vgf_\mu$ is the eigenfunction associated to the first eigenvalue $\gl_\mu$ of $-L_\mu$ in $\Gw$ (see Section2.1). By proceeding as the proof of (3.46) in \cite[Lemma 3.8]{GkV}, we obtain
\begin{equation}\label{N10-1}
\left(\int_{\prt\Gw}\eta d\gn\right)^{q'}\leq
\int_{\Gw}|\nabla u|^q\gz dx+\xl_\xm\int_{\Gw}u\gz dx+q'\left(\int_{\Gw} u^q \vgf_\xm^{-\frac{q}{\xa}}\gz dx\right)^{\frac{1}{q}}\left(\int_{\Gw}L[\eta]^{q'}dx\right)^{\frac{1}{q'}},
\end{equation}
where
 \begin{equation}\label{N10-2}
 L[\eta]=\left(2\vgf_\gm^{\frac{1}{\xa}-\frac{1}{q}}|\nabla\vgf_\gm.\nabla R[\eta]|+\vgf_\gm^{1+\frac{1}{\xa}-\frac{1}{q}}|\Gd R[\eta]|\right).
\end{equation}
Following the arguments of the proof of (3.48) in \cite[Lemma 3.9]{GkV} we can obtain
\begin{equation}\label{N10-1'}
 \int_{\Gw} L[\eta]^{q'}dx \leq c\|\eta\|^{q'-1}_{L^\infty(\prt\Gw)}\|\eta\|_{B^{\frac{\xa+1}{q}-\xa,q'}(\prt\Gw)}.
\end{equation}
By Lemma \ref{aniso1} we have
\begin{align}
\int_{\Gw}u^q \vgf_\xm^{-\frac{q}{\xa}}\gz dx\leq C\|\eta\|^{q'}_{L^\infty(\prt\Gw)}\int_{\Gw}\gd^{\ga-q} u^q dx\leq C\|\eta\|^{q'}_{L^\infty(\prt\Gw)}\left(1+\int_{\Gw}|\nabla u|^q\xd^\xa dx\right),\label{anis333}
\end{align}
where the constant $C$ depends on $\xO,\mu$ and $N.$

Combining \eqref{N10-1}, \eqref{N10-1'} and \eqref{anis333} we obtain
\bel{N10-13} \BAL
&\left(\int_{\prt\Gw}\eta d\gn\right)^{q'}\leq
\int_{\Gw}|\nabla u|^q\gz dx+\xl_\xm\int_{\Gw}u\gz dx\\
&\;\;\;\;\;\;\;\;\;\;\;+C\|\eta\|^{\frac{q'}{q}}_{L^\infty(\prt\Gw)}\left(1+\int_{\Gw}|\nabla u|^q\xd^\xa dx\right)^{\frac{1}{q}}\left(\|\eta\|^{q'-1}_{L^\infty(\prt\Gw)}\|\eta\|_{B^{\frac{\xa+1}{q}-\xa,q'}(\prt\Gw)}\right)^{\frac{1}{q'}}.
\EAL \ee

Let $K\subset\prt\Gw$ be a compact set. If $C^{\BBR^{N-1}}_{\frac{\xa+1}{q}-\xa,q'}(K)=0$ then there exists a sequence $\{\eta_n\}$ in $C^2_0(\prt\Gw)$ with the following properties:
  \begin{equation}\label{N20-1}
  0\leq\eta_n\leq 1\,,\; \eta_n=1\; \text{in a neighborhood of } K \text{ and } \lim_{n \to \infty}\eta_n = 0 \text{  in }  B^{\frac{1+\xa}{q}-\xa,q'}(\prt\Gw).
\end{equation}
This implies that $0 \leq R[\eta_n] \leq 1$ and $\lim_{n \to \infty}R[\eta_n]=0$ a.e. in $\Gw$. Put $\gz_n=\vgf_\gm R[\eta_n]^{q'}$. Then
\bel{lum1} \lim_{n \to \infty}\int_{\Gw}|\nabla u|^q \zeta_n dx = 0 \quad \text{and} \quad \lim_{n \to \infty}\int_{\Gw}u \zeta_n dx = 0. \ee
From \eqref{N10-13}--\eqref{lum1}, we obtain
$$ \nu(K) \leq \int_{\prt \Gw}\eta_n d \nu \to 0 \quad \text{as } n \to \infty. $$
This implies that $\nu$ is absolutely continuous with respect to $C^{\BBR^{N-1}}_{\frac{\xa+1}{q}-\xa,q'}$.

\noindent \textbf{Case 2:} $q=\xa+1.$ Let $0<\xe<\xa+1$ and $\gz=\vgf_\gm R[\eta]^{q'}.$ Proceeding as the proof of \eqref{N10-1}, we can prove
\bel{N10-1b} \BAL
\left(\int_{\prt\Gw}\eta d\gn\right)^{\frac{\xa+1}{\xa}}&\leq
\int_{\Gw}|\nabla u|^{\xa+1}\gz dx+\xl_\xm\int_{\Gw}u\gz dx\\
&+\frac{\xa+1}{\xa}\left(\int_{\Gw}u^{\xa+1}\vgf_\xm^{-\frac{\xa+1-\xe}{\xa}}\gz dx\right)^{\frac{1}{\xa+1}}\left(\int_{\Gw}L[\eta]^{\frac{\xa+1}{\xa}}dx\right)^{\frac{\xa}{\xa+1}},
\EAL \ee
where
 \begin{equation}\label{N10-2b}
 L[\eta]=\left(2\vgf_\gm^{\frac{\ga+1-\xe}{\xa q}-\frac{1}{q}}|\nabla\vgf_\gm.\nabla R[\eta]|+\vgf_\gm^{1+\frac{\ga+1-\xe}{\xa q}-\frac{1}{q}}|\Gd R[\eta]|\right).
\end{equation}

Using \eqref{aniso1b} and the ideas of the proof of \eqref{N10-13} we can obtain the following inequality

\bel{N10-13b} \BAL
&\left(\int_{\prt\Gw}\eta d\gn\right)^{\frac{\xa+1}{\xa}}\leq
\int_{\Gw}|\nabla u|^{\xa+1}\gz dx+\xl_\xm\int_{\Gw}u\gz dx\\&
\;\;\;\;\;\;\; + C\|\eta\|^{\frac{1}{\xa}}_{L^\infty(\prt\Gw)}\left(1+\int_{\Gw}|\nabla u|^{\xa+1}\xd^\xa dx\right)^{\frac{1}{\xa+1}}\left(\|\eta\|^{\frac{1}{\xa}}_{L^\infty(\prt\Gw)}\|\eta\|_{B^{\xe+1-\xa,\frac{\xa+1}{\xa}}(\prt\Gw)}\right)^{\frac{\xa}{\xa+1}},
\EAL \ee
where the constant $C$ depends on $N,\mu,\Gw$ and $\vge$.

The rest of the proof follows by using a similar argument as in the first case.
\end{proof}
\begin{proposition} \label{existbdrtr}
Let $u\in C^2(\xO)$ be a positive solution of \eqref{problem}. If $|\nabla u|\in L^q(\xO,\xd^\xa)$
 it possesses a boundary trace $\xn \in \mathfrak{M}(\partial\xO),$ i.e., $u$ is the solution of
the boundary value problem \eqref{N1} with this measure $\xn.$
\end{proposition}
\begin{proof}
If $v :=  \BBG_\mu^\xO[|\nabla u|^q]$ then $v\in L^1(\xO,\xd^\xa)$ and $u + v$ is a positive $L_{\xm }$ harmonic function.
Hence $u + v \in L^1(\xO,\xd^\xa)$ and there exists a non-negative measure $\xn \in \mathfrak{M}(\partial\xO)$ such
that $u + v =  \BBK_\mu^\xO[\xn].$ By Proposition \ref{PropA} we obtain the result.
\end{proof}

\begin{proof}[{\bf Proof of Theorem J}]
In view of the proof of \cite[Proposition A.2]{GkV} we can obtain the following estimates
 \be
|u(x)|\leq C_0\xd(x)^{\xa}\dist(x,K)^{-\frac{2-q}{q-1}-\xa}\qquad\forall x\in \xO,\label{3.4.24b}
  \ee
  \be
  |\nabla u(x)|\leq C_0\xd(x)^{\xa-1}\dist(x,K)^{-\frac{2-q}{q-1}-\xa}\qquad\forall x\in \xO,\label{3.4.24*b}
  \ee
where $C$ depends on $N,\mu,q,\Gw$ and $\sup _{\Gs_{\xb_0}}u$.

\noindent \textbf{Case 1:} Assume that $q\neq\xa+1$ and $C^{\BBR^{N-1}}_{\frac{\xa+1}{\xa}-\xa,q'}(K)=0$. Then there exists a sequence $\{\eta_n\}$ in $C^2_0(\prt\Gw)$ satisfying \eqref{N20-1}. In particular, there exists a decreasing sequence $\{\CO_n\}$ of relatively open subsets of $\prt\Gw$, containing $K$ such that $\eta_n=1$ on $\CO_n$ and thus $\eta_n=1$ on $K_n:=\overline\CO_n$.
We set $\tilde\eta_n=1-\eta_n$ and $\tilde \gz_n=\vgf_\gm R[\tilde\eta_n]^{2q'}$ where $R$ is defined by \eqref{N13}.
Then $0\leq\tilde\eta_n\leq 1$  and $\tilde\eta_n=0$ on $K_n$. Therefore
 \begin{equation}\label{N22}\BA {lll}
\tilde \gz_n(x)\leq \gf_\gm\min\left\{1, c\gd(x)^{1-N}e^{-(4\gd(x))^{-2}(\dist (x,K^c_n))^2}\right\}.
   \EA\end{equation}
Furthermore
\begin{equation}\label{N22-1}\BAL
&|\nabla R[\tilde\eta_n]|\leq c\min\left\{1, \gd(x)^{-2-N}e^{-(4\gd(x))^{-2}(\dist (x,K^c_n))^2}\right\},\\[2mm]
&|\Gd R[\tilde\eta_n]|\leq c\min\left\{1, \gd(x)^{-4-N}e^{-(4\gd(x))^{-2}(\dist (x,K^c_n))^2}\right\}.
\EAL \end{equation}
Proceeding as the proof of (3.65) in \cite[Theorem 3.10]{GkV} we have
\begin{equation}\label{N23}
\int_{\Gw}(uL_\gm\tilde \gz_n+|\nabla u|^q\tilde\gz_n) dx=0.
\end{equation}
Using the expression of $L_\gm\tilde\gz_n$, we derive from  \eqref{N23} that
\begin{equation}\label{N25}\BAL
 \int_{\Gw}|\nabla u|^q\tilde \gz_ndx&=\int_{\Gw}(-\gl_\gm \vgf_\gm R[\tilde\eta_n]^{2q'}+4q'R[\tilde\eta_n]^{2q'-1}\nabla\vgf_\gm.\nabla R[\tilde\eta_n]\\
&+ 2q'R[\tilde\eta_n]^{2q'-2}\vgf_\gm(R[\tilde\eta_n]\Gd R[\tilde\eta_n] +(2q'-1)|\nabla R[\tilde\eta_n]|^2))u dx \\
& \leq c \left(\int_{\Gw}u^q \vgf_\xm^{-\frac{q}{\xa}} \tilde \gz_n dx \right)^{\frac{1}{q}} \left(\int_{\Gw}\tilde L[\eta_n]^{q'}dx\right)^{\frac{1}{q'}}
 \EAL \end{equation}
where
 \begin{equation}\label{N25+}
\tilde L[\eta]=\vgf_\gm^{\frac{1}{\xa}-\frac{1}{q}}|\nabla\vgf_\gm.\nabla R[\eta_n]|+\vgf_\gm^{1+\frac{1}{\xa}-\frac{1}{q}}|\Gd R[\tilde\eta_n]|+
\vgf_\gm^{1+\frac{1}{\xa}-\frac{1}{q}}|\nabla R[\tilde\eta_n]|^2.
\end{equation}
By proceeding as in the proof of (3.75) in \cite[Theorem 3.10]{GkV} we can prove
\begin{equation}\label{N30}
         \int_{\xO}|\nabla u|^q \vgf_\mu R[\tilde\eta_n]^{2q'} dx
         \leq C\|\eta_n\|_{B^{\frac{\xa+1}{q}-\xa,q'}(\prt \Gw)}\left(\int_{\xO}\gd^{\ga-q} u^q R[\tilde\eta_n]^{2q'}dx\right)^{\frac{1}{q}}.
       \end{equation}
Next we recall that $\xb_0$ is the constant in Proposition \ref{geo}.

\noindent \textbf{Claim: } There exists a positive constant $\xb_1=\xb_1(q,\mu,N,\xO) \in (0,\gb_0)$ such that
\be\label{aniso1g}
\int_{\xO_{\xb_1}}\gd^{\ga-q}u^q R[\tilde\eta_n]^{2q'}dx\leq C\Big(\int_{\xO_{\xb_1}}\xd^\xa|\nabla u|^qdx+1+\|\eta_n\|_{B^{\frac{\xa+1}{q}-\xa,q'}(\prt \Gw)}^{q'}\Big),
\ee
where $C$ depends only on $N,\mu,\Gw$  and $\max(\sup_{ \Gs_{\xb_1}}u,\sup_{\Gs_{\xb_0}}u)$.

Indeed, let $\xb_1 \in (0,\xb_0)$.  By integration by parts, we have
\begin{align*}
&\int_{D_\xb\setminus D_{\xb_1}}\gd^{\ga-q} u^q R[\tilde\eta_n]^{2q'}dx=\frac{1}{\xa+1-q}\int_{D_\xb\setminus D_{\xb_1}}\nabla \xd^{\xa+1-q}\nabla \xd u^q R[\tilde\eta_n]^{2q'}dx\\
&=\frac{1}{\xa+1-q}\int_{D_\xb\setminus D_{\xb_1}}\Big(- \xd^{\xa+1-q} \xD\xd u^q R[\tilde\eta_n]^{2q'}-q \xd^{\xa+1-q}u^{q-1} R[\tilde\eta_n]^{2q'} \nabla\xd\nabla u
\\&\qquad \qquad \qquad\qquad \qquad -(2q'-1) \xd^{\xa+1-q}u^{q} R[\tilde\eta_n]^{2q'-1} \nabla\xd \nabla R[\tilde\eta_n]\Big)dx
\\&\;\;+\frac{1}{\xa+1-q}\Big(\int_{\Gs_{\xb_1}} \xd^{\xa+1-q} \frac{\partial\xd}{\partial \mathbf{n}_{\xb_1}} R[\tilde\eta_n]^{2q'} u^qdS+\int_{\Gs_{\xb}} \xd^{\xa+1-q} \frac{\partial\xd}{\partial \mathbf{n}_{\xb}} R[\tilde\eta_n]^{2q'}u ^qdS\Big)\\
&\leq\frac{C}{|\xa+1-q|}\Big(\int_{D_\xb\setminus D_{\xb_1}} \xd^{\xa+1-q}u^qR[\tilde\eta_n]^{2q'}dx+\int_{D_\xb\setminus D_{\xb_1}} \xd^{\xa+1-q}u^{q-1} |\nabla u|R[\tilde\eta_n]^{2q'}dx\\
&+\int_{D_\xb\setminus D_{\xb_1}} \xd^{\xa+1-q}u^{q} R[\tilde\eta_n]^{2q'-1} |\nabla R[\tilde\eta_n]|dx+\xb_1^{\xa+1-q}\sup_{\Gs_{\xb_1}}u^q+\int_{\Gs_{\xb}} \xd^{\xa+1-q}u^qR[\tilde\eta_n]^{2q'}dS\Big),
\end{align*}
where $C=C(N,\mu,q,\Gw)$. Now we choose $\xb_1$ small enough such that
\begin{align*}
\frac{C}{|\xa+1-q|}\int_{D_\xb\setminus D_{\xb_1}} \xd^{\xa+1-q}u^q R[\tilde\eta_n]^{2q'}dx\leq \frac{1}{16}\int_{D_\xb\setminus D_{\xb_1}} \xd^{\xa-q}u^qR[\tilde\eta_n]^{2q'}dx.
\end{align*}
By H\"{o}lder inequality we have
\begin{align*}
\frac{C}{|\xa+1-q|}\int_{D_\xb\setminus D_{\xb_1}}& \xd^{\xa+1-q}u^{q-1} |\nabla u|R[\tilde\eta_n]^{2q'}dx
\\
&\leq \frac{1}{16}\int_{D_\xb\setminus D_{\xb_1}}\xd^{\xa-q}u^{q}R[\tilde\eta_n]^{2q'}dx+C_1\int_{D_\xb\setminus D_{\xb_1}} \xd^{\xa} |\nabla u|^q R[\tilde\eta_n]^{2q'}dx
\end{align*}
where $C_1=C_1(N,\mu,q,\Gw)$. In view of the proof of  (3.53) in \cite[Lemma 3.9]{GkV}, by \eqref{3.4.24b} and  H\"{o}lder inequality, we obtain
\begin{align*}
&\frac{C}{|\xa+1-q|}\int_{D_\xb\setminus D_{\xb_1}} \xd^{\xa+1-q}u^{q} R[\tilde\eta_n]^{2q'-1} |\nabla R[\tilde\eta_n]|dx\\
&\leq \frac{C C_0}{|\xa+1-q|}\int_{D_\xb\setminus D_{\xb_1}} \xd^{\xa-1} u R[\tilde\eta_n]^{2q'-1} |\nabla R[\tilde\eta_n]|dx\\
&\leq \frac{1}{16}\int_{D_\xb\setminus D_{\xb_1}} \xd^{\xa-q}u^{q} R[\tilde\eta_n]^{2q'}dx + C_2 \int_{D_\xb\setminus D_{\xb_1}} \xd^{\xa} | \nabla R[\eta_n]|^{q'}dx\\
&\leq\frac{1}{16}\int_{D_\xb\setminus D_{\xb_1}} \xd^{\xa-q}u^{q} R[\tilde\eta_n]^{2q'}dx+C_3\|\eta_n\|_{B^{\frac{\xa+1}{q}-\xa,q'}(\prt \Gw)}^{q'}
\end{align*}
where $C_i=C_i(N,\mu,q,\Gw,C_0)$, $i=2,3$. Combining all above we can easily deduce
\begin{equation} \label{ani11} \BAL
&\int_{D_\xb\setminus D_{\xb_1}}\gd^{\ga-q} u^q R[\tilde\eta_n]^{2q'}dx\\
\leq &C\Big(\int_{D_\xb\setminus D_{\xb_1}}\xd^\xa|\nabla u|^qdx+1+\|\eta_n\|_{B^{\frac{\xa+1}{q}-\xa,q'}(\prt \Gw)}^{q'}+\int_{\Gs_{\xb}} \xd^{\xa+1-q}u^qR[\tilde\eta_n]^{2q'}dS\Big),
\EAL \end{equation}
where $C$ depends only on  $N,\mu,q,\Gw$ and $\max(\sup_{ \Gs_{\xb_1}}u,\sup_{ \Gs_{\xb_0}}u)$.

Now by \eqref{3.4.24b} and \eqref{N22}, we have
$$\lim_{\xb\rightarrow0}\int_{\Gs_{\xb}} \xd^{\xa+1-q}u^qR[\tilde\eta_n]^{2q'}dS=0,$$
hence letting $\xb \to 0$ in \eqref{ani11}, we obtain the claim.

Combining \eqref{3.4.24b}, \eqref{N30} and \eqref{aniso1g} leads to
\begin{align*}
\int_{\Gw_{\gb_1}}|\nabla u|^q \ei R[\tilde\eta_n]^{2q'} dx
         \leq C\|\eta_n\|_{B^{\frac{\xa+1}{q}-\xa,q'}(\prt \Gw)}\left(\int_{\Gw_{\gb_1}}\xd^\xa|\nabla u|^qdx+1+\|\eta_n\|_{B^{\frac{\xa+1}{q}-\xa,q'}(\prt \Gw)}^{q'}\right)^{\frac{1}{q}},
\end{align*}
which implies
\begin{align*}
\int_{\Gw_{\gb_1}}|\nabla u|^q(R[\tilde\eta_n])^{2q'}\ei dx
         \leq C\left(\|\eta_n\|_{B^{\frac{\xa+1}{q}-\xa,q'}(\prt \Gw)}^{q'}+\|\eta_n\|_{B^{\frac{\xa+1}{q}-\xa,q'}(\prt \Gw)} +1 \right),
\end{align*}
Letting $n\to\infty$ and using the fact that $\eta_n\to 0$, we obtain by Fatou's lemma that
       $$ \int_{\Gw_{\frac{\gb_1}{2}}}|\nabla u|^q\gf_\mu dx\leq C<\infty.$$
       Combining this with the fact that $|\nabla u|$ is bounded in $D_{\frac{\gb_1}{2}}$ due to \eqref{3.4.24*b},  we assert that $|\nabla u|\in L^q(\xO,\xd^\xa).$ Thus by Proposition \ref{existbdrtr}  there exists a nonnegative Radon measure $\xn$ with support in $K$ such that
       $$u+\BBG_\mu^\xO[|\nabla u|^q]=\BBK_\mu^\xO[\xn].$$
 In light of Theorem I, $\xn\equiv0$ which implies $u=0$ and the result follows in this case.

\noindent \textbf{Case 2:} Assume that $q=\xa+1$ and $C^{\BBR^{N-1}}_{\xe+1-\xa,q'}(K)=0$ for $\vge$ as in statement (ii).  Then we can obtain the desired result by combining the ideas in Case 1 of this theorem and in Case 2 of Theorem I.
\end{proof}

\appendix \section{Barrier}
\setcounter{equation}{0}

In this section we will construct a barrier which plays an important role. This barrier will have the same properties as the barrier in \cite[Proposition 6.1]{GkV}. Let $\gb_0$ be the constant in Proposition \ref{geo}.

\begin{proposition}\label{barr} Let $\Gw\subset\BBR^N$ be a $C^2$ domain, $0<\xm \leq\frac{1}{4}$ and $q>1$. Then for any $z\in \prt\Gw$ and $0<R\leq \frac{\xb_0}{16}$, there exists a super solution $w:=w_{z,R}$ of $(\ref{Epower})$ in $\Gw\cap B_R(z)$ such that $w\in C(\overline\Gw\cap B_R(z))$, $w(x)\to \infty$ when $\dist (x,K)\to 0$, for any compact subset $K\subset \Gw\cap \prt B_R(z)$ and $w$ vanishes on $\prt\Gw\cap B_R(z)$. More precisely
\begin{equation}\label{BAR1}
w(x)= \left\{\BAL &c(R^2-|x-z|^2)^{-b}\gd(x)^{\gamma} \qquad\forall \gamma\in (1-\xa,\xa)\,&&\text{ if }0<\gm<\frac{1}{4}\\
&c(R^2-|x-z|^2)^{-b}\gd(x)^{\frac{1}{2}}\left(\ln \frac{{\rm {diam}}(\Gw)}{\gd(x)}\right)^{\frac{1}{2}}  &&\text{ if }\gm=\frac{1}{4}
\EAL \right.
\end{equation}
where $b\geq \max\{\frac{4-q}{q-1}+\gg,\frac{N-2}{2},1\}$ and $c=c(N,\mu,q,b,\gg)$.
\end{proposition}  
\begin{proof} Take $z \in \prt \Gw$ and $R \in (0,\frac{\gb_0}{16}]$. Without loss of generality, we may  assume  $z=0$.

\noindent \textbf{Case 1:} $\gm <\frac{1}{4}$. Set
$$ w(x)=\Gl(R^2-| x|^2)^{-b} \gd(x)^{\gamma} \quad x \in \Gw \cap B_R(0), $$
where $b,\gamma>0$ to be chosen later on. By straightforward calculation we have
\be\label{ineq111}\BAL
\xL^{-q}|\nabla w|^q &=\xL^{-q}(|\nabla w|^2)^\frac{q}{2} =\gd^{q(\xg-1)} (R^2-|x|^2)^{-q(b+1)} \times \\
& \times\left(\xg^2(R^2-|x|^2)^2+4b^2\gd^2|x|^2+4b \gg \gd (R^2-|x|^2) x\cdot\nabla \gd \right)^{\frac{q}{2}}.
\EAL \ee

Now let $x_0=x+(\frac{\xb_0}{4}-\gd(x))\nabla \gd(x)$ and $\xi_x$ be the unique point in $\partial\xO$ such that $\gd(x)=|x-\xi_x|.$ Then $\gd(x_0)=|x_0-\xi_x|=\frac{\xb_0}{4},$
$$B(x_0,\frac{\xb_0}{4})\subset\xO\quad\text{and}\quad \overline{B}(x_0,\frac{\xb_0}{4})\cap \partial\xO=\{\xi_x\}.$$
Furthermore
$$\nabla \gd(x)=\frac{x_0-x}{|x-x_0|}.$$
We have
$$ \BAL |x-x_0|^2\leq \frac{\xb_0^2}{16}\Longleftrightarrow |x|^2+|x_0|^2-2x\cdot x_0\leq \frac{\xb_0^2}{16}\Longleftrightarrow|x|^2+|x_0|^2-\frac{\xb_0^2}{16}\leq 2x\cdot x_0.
\EAL $$
Since $|x_0|^2\geq \frac{\xb_0^2}{16}$, it follows from the last inequality that
$x\cdot x_0\geq0$. Therefore
\be
x\cdot\nabla \gd(x)=\frac{x\cdot x_0-|x|^2}{|x-x_0|}\geq-\frac{|x|^2}{|x-x_0|}\geq -\frac{|x|}{3},\label{inq11}
\ee
where in the above inequality we have used the fact that $|x|\leq\frac{\xb_0}{16}$ and $|x-x_0|\geq\frac{3\xb_0}{16}$. By \eqref{inq11} and H\"{o}lder inequality we have
\bel{inq1111} \BAL
&\xg^2(R^2-|x|^2)^2+4b^2\gd(x)^2|x|^2+4b\xg (R^2-|x|^2) \gd(x) x\cdot\nabla \gd(x)\\
&\geq \frac{2}{3}\left(\xg^2(R^2-|x|^2)^2+4b^2\gd(x)^2|x|^2\right)\geq \frac{2\xg^2}{3}(R^2-|x|^2)^2.
\EAL \ee
Combining \eqref{inq1111} and \eqref{ineq111} we have
\be
\xL^{-q}|\nabla w(x)|^q\geq \left(\frac{2\xg^2}{3}\right)^{\frac{q}{2}} (R^2-|x|^2)^{-qb}  \gd(x)^{q(\xg-1)}.
\ee
The rest proof is similar to the proof of \cite[Proposition 6.1]{GkV} and we omit it.
 \medskip

\noindent \textbf{Case 2:} $\xm =\frac{1}{4}$. Set
$$ w(x):=\Gl(R^2-|x|^2)^{-b}\gd^{\frac{1}{2}}\left(\ln \frac{e^2R}{\gd}\right)^\frac{1}{2} \quad x \in \Gw \cap B_R,$$
for some $\Gl,b$ to be made precise later on. It is easy to calculate
\bel{fe5} \nabla w = (R^2-r^2)^{-b-1}\gd^{-\frac{1}{2}}  \Big(\ln \frac{e^2R}{\gd} \Big)^{\frac{1}{2}} \Big[ \frac{1}{2}(R^2-r^2) \Big(1-\Big(\ln\frac{e^2R}{\gd}\Big)^{-1}\Big)\nabla \gd+2b\gd x \Big].
\ee
By \eqref{inq11} and H\"{o}lder inequality, we have
$$\BAL
&\Big|\frac{1}{2}(R^2-r^2)\Big(1-\Big(\ln\frac{e^2R}{\gd}\Big)^{-1}\Big)\nabla \gd+2b\gd x\Big|^2\\
&=\frac{1}{4}\Big[(R^2-r^2)\Big(1-\Big(\ln\frac{e^2R}{\gd}\Big)^{-1}\Big)\Big]^2
+4b^2\gd^2|x|^2+2b(R^2-r^2)\Big(1-\Big(\ln\frac{e^2R}{\gd}\Big)^{-1}\Big)\gd x \cdot \nabla \gd\\
&\geq \frac{1}{4}\Big[(R^2-r^2)\Big(1-\Big(\ln\frac{e^2R}{\gd}\Big)^{-1}\Big)\Big]^2+4b^2\gd^2|x|^2
-\frac{2}{3}b(R^2-r^2)\Big(1-\Big(\ln\frac{e^2R}{\gd}\Big)^{-1}\Big)\gd |x|\\
&\geq \frac{1}{12}\Big[(R^2-r^2)\Big(1-\Big(\ln\frac{e^2R}{\gd}\Big)^{-1}\Big)\Big]^2\\
&\geq \frac{1}{48}(R^2-r^2)^2.
\EAL
$$
Hence,
\be
\xL^{-q}|\nabla w|^q\geq \left(\frac{1}{48}\right)^{\frac{q}{2}} (R^2-r^2)^{-qb} \gd^{-\frac{q}{2}}\Big(\ln\frac{e^2R}{\gd}\Big)^{\frac{q}{2}} .
\ee
Proceeding as the proof of \cite[Proposition 6.1]{GkV}, we obtain the desired result.
\end{proof}


\begin{thebibliography}{99}

\bibitem{Ad} D.R. Adams, L.I. Hedberg, Function Spaces and Potential Theory, Springer, New York, 1996.
%
\bibitem{An2} A. Ancona, {\em Negatively curved manifolds, elliptic operators  and the Martin boundary}, Ann. of Math. (2), {\bf 125} (1987), 495-536.
%
\bibitem{BMR} C. Bandle, V. Moroz and W. Reichel, {\em Boundary blowup type    sub-solutions to semilinear elliptic equations with Hardy potential}, J. London Math. Soc. {\bf 2} (2008), 503--523.
%
\bibitem{BMM} C. Bandle, M. Marcus and V. Moroz, \emph{Boundary singularities of solutions to elliptic equations  with Hardy potential}, Israel J. Math. {\bf 222} (2017), 487-514.
%
\bibitem{BGV} M. F. Bidaut-V\'eron, M. Garcia-Huidobro and L. V\'eron, {\em Boundary singularities of positive solutions of quasilinear Hamilton-Jacobi equations}, Calc. Var. Partial Differential Equations {\bf 54} (2015), no. 4, 3471--3515.
%
\bibitem{BHNV} M. F. Bidaut-V\'eron, G. Hoang, Q. H. Nguyen, L. V\'eron, \emph{An elliptic semilinear equation with source term and boundary measure data: the supercritical case}, Journal of Functional Analysis {\bf 269} (2015), 1995-2017.
%
\bibitem{BVi} M. F. Bidaut-V\'eron and L. Vivier, {\em An elliptic semilinear
equation with source term involving boundary measures: the subcritical case},
Rev. Mat. Iberoamericana {\bf 16} (2000), 477-513.
%
\bibitem{dAdP} L. D'Ambrosio and S. Dipierro, \emph{Hardy inequalities on Riemannian manifolds and applications},  Ann. Inst. H. Poincar\'e Anal. Non Lin\'eaire {\bf 31} (2014), 449--475.

\bibitem{FMT} S. Filippas, L. Moschini and A. Tertikas, {\em Sharp two-sided heat kernel estimates for critical Schrodinger operators on bounded domains}, Commun. Math. Phys. {\bf 273} (2007), 237-281.
%
\bibitem{GkNg} K. T. Gkikas and P. T. Nguyen, \emph{On the existence of weak solutions of semilinear elliptic equations and systems with Hardy potentials},  Journal of Differential Equations {\bf 266} (2019), 833--875.
%
\bibitem{GkV} K. T. Gkikas and L. V\'eron, \emph{Boundary singularities of solutions of semilinear elliptic equations with critical Hardy potentials}, Nonlinear Anal. \textbf{121} (2015), 469--540.
%
\bibitem{GT} D. Gilbarg and N. Trudinger, {\em Elliptic Partial Differential Equations of Second Order.} (Second edition), Springer, Berlin (1983).
%
\bibitem{Ka} J. L. Kazdan and R. J. Kramer, {\em{Invariant criteria for existence of solutions to second order quasilinear elliptic equations},} Comm. Pure Appl. Math. {\bf 31} (1978), 619-645.
%
\bibitem{La} O.A. Ladyzhenskaya and N. N.  Ural\'tseva,  {\em Linear and quasilinear elliptic equations}. Translated from the Russian by Scripta Technica, Inc. Translation editor: Leon Ehrenpreis Academic Press, New York-London (1968).
%
\bibitem{MaMo} M. Marcus and V. Moroz, {\em Moderate solutions of semilinear elliptic equations with Hardy potential under minimal restrictions on the potential}, Ann. Sc. Norm. Super. Pisa Cl. Sci. {\bf 18} (2018), no. 1, 39--64.
%
\bibitem{MaNg} M. Marcus and P. T. Nguyen, {\em Moderate solutions of semilinear elliptic equations with Hardy potential}, Ann. Inst. H. Poincar\'e Anal. Non Lin\'eaire {\bf 34} (2017), 69--88.
%
\bibitem{MaNg2} M. Marcus and P. T.  Nguyen, {\em Elliptic equations with nonlinear absorption depending on the solution and its gradient}, Proc. Lond. Math. Soc. {\bf 111} (2015), 205--239.
%
\bibitem{MVbook} M. Marcus and L. V\'eron, {\em Nonlinear second order elliptic equations involving measures}, De Gruyter Series in Nonlinear Analysis and Applications, 2013.
%
\bibitem{Ng1} P. T. Nguyen, {\em Isolated singularities of positive solutions of elliptic equations with weighted gradient term}, Anal. PDE {\bf 9} (2016), 1671--1692.
%
\bibitem{Ng2} P. T. Nguyen, {\em Semilinear elliptic equations with Hardy potential and subcritical source term}, Calc. Var. Partial Differential Equations {\bf 56} (2017), no. 2, Art. 44, 28 pp.
%
\bibitem{NgVe1} P. T. Nguyen and L. V\'eron, {\em Boundary singularities of solutions to elliptic viscous Hamilton-Jacobi equations}, J. Funct. Anal. {\bf 263} (2012), no. 6, 1487-1538.
%
\bibitem{Stein} E. M. Stein, \em{Singular integrals and differentiability properties of funcions}, Princeton University Press,
1970.
%
\end{thebibliography}
\end{document}